\pdfoutput=1
\documentclass[12pt]{article}
\usepackage{amsmath,latexsym,amssymb,amsthm,enumerate,amsmath,amscd}
\usepackage[mathscr]{eucal}
\usepackage{color}
\usepackage[usenames,dvipsnames]{xcolor}
\usepackage{colortbl}
\usepackage[all,cmtip]{xy}
\usepackage{graphicx}
\usepackage{comment}
\definecolor{light-gray}{gray}{0.9}
\definecolor{med-gray}{gray}{0.5}
\definecolor{gray1}{gray}{0.87}
\definecolor{gray2}{gray}{0.74}
\definecolor{gray3}{gray}{0.64}
\definecolor{gray4}{gray}{0.55}
\definecolor{verylight-yellow}{rgb}{1,1,0.7}
\definecolor{yellow}{rgb}{1,1,0.2}
\definecolor{vivid-blue}{rgb}{0.2,0,1}
\definecolor{light-pink}{rgb}{1,0.8,1}
\definecolor{med-pink}{rgb}{1,0.6,1}
\definecolor{aqua}{rgb}{0.0, 1.0, 1.0}
\definecolor{light-gray}{rgb}{0.5, 0.9, 0.5}

\theoremstyle{plain}
\newtheorem{theorem}{Theorem}[section]
\newtheorem{proposition}[theorem]{Proposition}
\newtheorem{corollary}[theorem]{Corollary}
\newtheorem{lemma}[theorem]{Lemma}
\theoremstyle{definition}
\newtheorem{definition}[theorem]{Definition}
\newtheorem{conjecture}[theorem]{Conjecture}
\newtheorem{remark}[theorem]{Remark}
\newtheorem{example}[theorem]{Example}

\newtheorem{thm}{Theorem}

\newtheorem{cor}[theorem]{Corollary}

\newtheorem{conj}[theorem]{Conjecture}

\newtheorem{rem}[thm]{Remark}

\newtheorem*{ack}{Acknowledgment}

\numberwithin{equation}{section}
\numberwithin{table}{section}
\definecolor{purple}{rgb}{0.4,0.2,0.4}

\def\cha{\mathrm{char}\ }

\def\Mat{\mathrm{Mat}}

\def\<{\left<}
\def\>{\right>}

\def\max{\mathrm{max}}

\title{Nilpotent matrices having a given Jordan type as maximum commuting nilpotent orbit}

\author{Anthony Iarrobino, Leila Khatami, Bart Van Steirteghem, Rui Zhao\\[.05in]
}

\date{July 28, 2016, revised July 10, 2017 and February 3, 2018}

\begin{document}
\maketitle
\begin{abstract}
The Jordan type of a nilpotent matrix is the partition giving the sizes of its Jordan blocks. We study pairs of partitions $(P,Q)$, where $Q=\mathfrak Q(P)$ is the Jordan type of a generic nilpotent matrix $A$ commuting with a nilpotent matrix $B$ of Jordan type $P$.   T.~Ko\v{s}ir and P. Oblak have shown that $Q$ has parts that differ pairwise by at least two. Such partitions, which are also known as ``super distinct'' or ``Rogers-Ramanujan'',
are exactly those that are stable or ``self-large'' in the sense that $\mathfrak Q(Q)=Q$. \par
In 2012 P. Oblak formulated a conjecture concerning the cardinality of $\mathfrak Q^{-1}(Q)$ when $Q$ has two parts, and proved some special cases. R. Zhao
refined this to posit that the partitions in $\mathfrak Q^{-1}(Q)$ for $ Q=(u,u-r)$ with $ u>r>1 $ could be arranged in  an $(r-1)\times (u-r)$ table $\mathcal T(Q)$ where the entry in the
$k$-th row and $\ell$-th column has $k+\ell$ parts.
 We prove this Table Theorem, and then generalize the statement to propose a Box Conjecture  for  the set of partitions $\mathfrak Q^{-1}(Q)$ for
  an arbitrary partition $Q$ whose parts differ pairwise by at least two.
\end{abstract}
\tableofcontents
\section{Introduction.}
We fix an infinite field $\sf k$ and denote by $\Mat_n(\sf k)$ the ring of $n\times n$ matrices with entries in $\sf k$ acting on the vector space $V={\sf k}^n$. Let $P$ be a partition of $n$ and denote by $B=J_P$ the nilpotent Jordan block matrix of partition $P$. 
Let $\mathcal C_B=\{A\in \Mat_n({\sf k})\mid  AB=BA\}$ be the centralizer of $B$ in $\Mat_n(\sf k)$, and let $\mathcal N_B$ be the subvariety of nilpotent elements in $\mathcal C_B$. \par
There has been substantial work in the last ten years studying the map $\mathfrak Q$ that takes $P$ to the Jordan type $\mathfrak Q(P)$ of a generic element of $\mathcal N_B$. P.~Oblak conjectured a beautiful recursive description of $\mathfrak Q(P)$. This conjecture remains open in general (for progress on it see Section~\ref{recursiveconjsec},
 Conjecture~\ref{Oblakrecursconj}, Remark \ref{historyrem}, and \cite{Bas2,BKO,IK,Kh1,Kh2,Obl1}).\par
 An {\emph{almost rectangular} partition is one whose largest part is at most one larger than its smallest part. R. Basili introduced the invariant $r_P$,\label{rppage} which is the smallest number of almost rectangular partitions whose union is $P$, and showed that $\mathfrak Q(P)$ has $r_P$ parts (Theorem \ref{rpthm}).  T. Ko\v{s}ir and P. Oblak showed that if the characteristic of ${\sf k}$ is $0$ then $\mathfrak Q(P)$ has parts that differ pairwise by at least two (Theorem~\ref{stablethm}). Even in cases where the Oblak recursive conjecture had been shown  some time ago, (as $r_P= 2$ \cite{KO}, or $r_P=3$ \cite{Kh2}) the set $\mathfrak Q^{-1}(Q)$ remained  mysterious.  In 2012 P.~Oblak made a second conjecture: when $Q=(u,u-r)$ with $u> r\ge 2 $, then the cardinality  $|\mathfrak Q^{-1}(Q)|=(r-1)(u-r)$ \cite[Remark 2]{Obl2}. In 2013, R.~Zhao noticed an even stronger pattern in $\mathfrak Q^{-1}(Q)$ for such $Q$. She conjectured that there is a table $\mathcal T(Q)$ of partitions $P_{k,\ell}$ where the number of parts in $P_{k,\ell}$ is $k+\ell$: see Theorem \ref{Zthm} immediately below.  We here prove a precise version, the Table Theorem (Theorems \ref{tablethm} and \ref{table2thm}).  We then propose a Box Conjecture~\ref{Zgenconj} describing $\mathfrak Q^{-1}(Q)$ for arbitrary partitions $Q$ whose parts differ pairwise by at least two (Section \ref{boxconjsubsec}),
and we study some special cases where $Q$ has three parts (Section~\ref{boxspecialsec}).
\vskip 0.2cm
 
 The question, which pairs of conjugacy classes can occur for pairs of commuting matrices reduces to the case where both matrices are nilpotent. There is an extensive literature on commuting pairs of nilpotent matrices,  including \cite{Bas2,GurSe,IK,Kh1,Kh2,KO,Obl1,Obl2,Pan,Prem} and  others, some of whose results we specifically cite.  Connections to the Hilbert scheme are made in 
 \cite{Bar,Bas,BI,BuEv,HaHy,Nak,Prem}, and  commuting nilpotent orbits occur in the study of Artinian algebras \cite{BI,HW}.  However, the study of the map $P\to \mathfrak Q(P)$  seems to be, surprisingly, very recent, beginning with \cite{Bar,Bas,BI,Kh1,KO,Obl1,Pan, Prem}: apparently, early workers in the area were more drawn to determining vector spaces of commuting matrices of maximum dimension (see \cite{J,Ma,SuTy} and references in the latter). There is further recent work on commuting $r$-tuples of nilpotent matrices, as \cite{GurSe,NgSi,Si} and these also appear to be connected to the study of group schemes \cite{FPS,Ng,SusFB1,SusFB2}. There is much study of nilpotent orbits for Lie algebras, as in \cite{BrBr,CM,Gi,Pan2}; for generalizations of problems considered here to other Lie algebras than $sl_n$, see \cite{Pan}.\vskip 0.2cm
 
Our main result is 
\begin{theorem}\label{Zthm} Let $Q=(u,u-r)$ where $u>r\ge 2$.\begin{enumerate}[i.]
\item The cardinality $| \mathfrak Q^{-1}(Q)|=(r-1)(u-r)$.
\item The set $\mathfrak Q^{-1}(Q)$ may be arranged as an $(r-1)\times (u-r)$ array $\mathcal T(Q)$ of partitions 
\begin{equation}
P_{k,\ell}=P_{k,\ell}(Q), \text { where }1\le k \le r-1, \text { and } 1\le \ell\le u-r,
\end{equation}
such that the number of parts of $P_{k,\ell}$ is $k+\ell$.
\end{enumerate}
\end{theorem}
\begin{remark}
We call this the {\sc Table Theorem}.  Theorem \ref{tablethm} below  specifies each partition $P_{k,\ell}$ in the array or table $\mathcal T(Q)$, and shows that $\mathfrak Q(P_{k,\ell})=Q$; the Completeness Theorem \ref{table2thm} says that $\mathcal T(Q)$ is all of $\mathfrak Q^{-1}(Q)$. Our main tool is P. Oblak's result giving the largest part of $\mathfrak Q(P)$, see \cite{Obl1} and Theorem \ref{indexthm} below.  Some special cases had been shown prior to our work here:  P.~Oblak had shown Theorem \ref{Zthm} for  $2\le r\le 4$ in \cite{Obl2}. R. Zhao in \cite{Z} had shown the case $(u-r)\in\{1,2,3\}$ and also the case $ u\gg r$.

In the formula for cardinality the proposed value  for $|{\mathfrak Q}^{-1}((u,u-r))|$
is the same as that for $|\mathfrak Q^{-1}((u,r-1))|$. Understanding this symmetry was a goal of R.~Zhao in her study  of the two sets: it remains obscure to us. 
\end{remark}
We illustrate Theorem \ref{Zthm} in Example \ref{hook1ex} below (see also Example \ref{27,3ex}). We first introduce almost rectangular partitions\footnote{This name, alluding to the Ferrers diagram (see Figure \ref{5slmostrectfig}), was suggested by T. Ko\v{s}ir and P. Oblak.}, whose importance for the problem of describing the map $P\to \mathfrak Q(P)$ was first noted by R.~Basili \cite{Bas}.
\begin{definition}[Almost Rectangular]\label{ARdef} A partition $P=(p_1,p_2,\ldots ,p_s)$ of $n$ with $ p_1\ge p_2\ge \cdots \ge p_s>0$ is \emph{almost rectangular} if  $p_1-p_s\le 1$. For $1\le k\le n$ we denote by  $[n]^k$ the unique almost rectangular partition of $n$ that has $k$ parts (see equation \eqref{nkeq}).
\end{definition}
\begin{table}[ht]
\large
$$\begin{array}{|l|l|l|}
\hline 

(13,3)&(13,[3]^2)&(13,[3]^3)\\
\hline
([13]^2,3)&([13]^2,[3]^2)&([13]^2,[3]^3)\\
\hline 
([13]^3,3)&([13]^3,[3]^2)&([13]^3,[3]^3)\\
\hline
{\cellcolor{light-gray}(5, [11]^{4})}&([13]^4, [3]^{2})&([13]^4, [3]^{3})\\
\hline
{\cellcolor{light-gray}(5, [11]^{5})}&{\cellcolor{med-pink}\color{blue}{([7]^2, [9]^{5})}}&([13]^5, [3]^{3})\\
\hline
{\cellcolor{light-gray}(5, [11]^{6})}&{\cellcolor{med-pink}\color{blue}{([7]^2, [9]^{6})}}&{\cellcolor{aqua}{\color{purple}([9]^{3}, [7]^{6})}}\\
\hline
{\cellcolor{light-gray}(5, [11]^{7})}&{\cellcolor{med-pink}\color{blue}{([7]^2, [9]^{7})}}&{\cellcolor{aqua}{\color{purple}([9]^{3}, [7]^{7})}}\\
\hline
{\cellcolor{light-gray}(5, [11]^{8})}&{\cellcolor{med-pink}\color{vivid-blue}{([7]^2, [9]^{8})}}&{\cellcolor{med-pink}\color{vivid-blue}{([7]^2, [9]^{9})}}\\
\hline
{\cellcolor{light-gray}(5, [11]^{9})}&{\cellcolor{light-gray}(5,[11]^{10})}&{\cellcolor{light-gray}(5,[11]^{11})}\\
\hline\end{array}$$
\caption{Table $\mathcal T(Q), Q=(13,3)$.}\label{13,3table}
\end{table}

\vskip 0.4cm
\begin{table}[ht]
\large
$$\begin{array}{|l|l|l|l|}
\hline 

(14,4)&(14,[4]^2)&(14,[4]^3)&(14,[4]^4)\\
\hline
([14]^2,4)&([14]^2,[4]^2)&([14]^2,[4]^3)&([14]^2,[4]^4)\\
\hline 
{\cellcolor{light-gray}(6, [12]^{3})}&([14]^3,[4]^2)&([14]^3,[4]^3)&([14]^3,[4]^4)\\
\hline
{\cellcolor{light-gray}(6, [12]^{4})}&([14]^4, [4]^{2})&([14]^4, [4]^{3})&([14]^4, [4]^{4})\\
\hline
{\cellcolor{light-gray}(6, [12]^{5})}&{\cellcolor{med-pink}{([8]^2, [10]^{5})}}&{\cellcolor{aqua}{([10]^{3}, [8]^{5})}}&([14]^5, [4]^{4})\\
\hline
{\cellcolor{light-gray}(6, [12]^{6})}&{\cellcolor{med-pink}{([8]^2, [10]^{6})}}&{\cellcolor{aqua}{([10]^{3}, [8]^{6})}}&{\cellcolor{yellow}{([12]^{4}, [6]^{6})}}\\
\hline
{\cellcolor{light-gray}(6, [12]^{7})}&{\cellcolor{med-pink}{([8]^2, [10]^{7})}}&{\cellcolor{aqua}{([10]^{3}, [8]^{7})}}&{\cellcolor{aqua}{([10]^{3}, [8]^{8})}}\\
\hline
{\cellcolor{light-gray}(6, [12]^{8})}&{\cellcolor{med-pink}{([8]^2, [10]^{8})}}&{\cellcolor{med-pink}{([8]^2, [10]^{9})}}&{\cellcolor{med-pink}{([8]^2, [10]^{10})}}\\
\hline
{\cellcolor{light-gray}(6, [12]^{9})}&{\cellcolor{light-gray}(6,[12]^{10})}&{\cellcolor{light-gray}(6,[12]^{11})}&{\cellcolor{light-gray}(6,[12]^{12})}\\
\hline\end{array}$$
\caption{Table $\mathcal T(Q), Q=(14,4)$.}\label{14,4table}
\end{table}
\begin{example}\label{hook1ex}  We illustrate Theorem \ref{Zthm} by giving the  table $\mathcal T(Q)$ for $Q=(13,3)$  (Table~\ref{13,3table}) and for $Q=(14,4)$ (Table~\ref{14,4table}). The entries of the table use the notation $[n]^k$ for the almost rectangular partition of $n$ having $k$ parts. The table $\mathcal T(Q)$ has a decomposition into ``A-rows'' and ``B/C hooks'' (see Definition \ref{tabledef}, and Remark \ref{tablerem}): both Table~\ref{13,3table}) and Table~\ref{14,4table}) show horizontal rows or partial rows comprised of type A partitions
(these cells of the table are unshaded), and  type B/C hooks (each hook is shaded or colored). For the definition of type A,B,C partitions see Definition~\ref{casesdef}. For $Q=(13,3)$ the table $\mathcal T(Q)$ has three B/C hooks: of the form $\{(5,[11]^k), 4\le k\le 11\}$, $\{([7]^2, [9]^k), 5\le k\le 9\}$, and $\{([9]^3,[7]^k), 6\le k\le 7\}$. For $Q=(14,4)$ the table $\mathcal T(Q)$ has four B/C hooks: the fourth B/C ``hook'' is comprised of a single cell
$P=([12]^4,[6]^6)$.
\end{example}
\subsubsection{Overview.} 
In Section \ref{notationprelimsec} we first review some results we will need.
In Section \ref{posetmapsec} we recall the poset $\mathcal D_P$ associated to the nilpotent commutator $\mathcal N_B$ of  $B=J_P$ and more particularly  to a maximal nilpotent subalgebra $\mathcal U_B$ of the centralizer $\mathcal C_B$.
Let $Q=(u,u-r)$ with $u>r\ge 2$ and put $B=J_Q$. After dividing the partitions in $\mathfrak Q^{-1}(Q)$ into three types A, B and C, in  Section \ref{threesubsetsec}, we prove in Section \ref{tablesec}  the main Theorem~\ref{tablethm} which specifies the filling of the table $\mathcal T(Q)$ with A rows and B/C hooks.  We also give some properties of the tables in Remark \ref{tablerem} and we display the $Q=(27,3)$ table in Example \ref{27,3ex} and Table \ref{27,3table}. We obtain in
 Corollary~\ref{normalpatterncor} the alternating pattern case first shown by R. Zhao \cite{Z}, which occurs for $u\gg r$.\footnote{These are called \emph{normal} patterns in \cite{IKVZ} and \cite{Z}.} In Section \ref{tablecompletenesssec} we show  that the table $\mathcal T(Q)$ is the complete inverse image of $Q$ under the map $\mathfrak Q$ (Theorem~\ref{table2thm}).   \par 
After reviewing P. Oblak's recursive conjecture in Section \ref{recursiveconjsec}, we propose in Section \ref{boxconjsubsec}  the Box Conjecture \ref{Zgenconj} for $\mathfrak Q^{-1}(Q)$. The combinatorial part of the Box Conjecture in short states
that if $Q$ is a partition with $k$ parts differing pairwise by at least two, then its key ${\sf S}_{Q}$ gives the lengths of the sides of a $k$-dimensional box $\mathcal B(Q)$ containing the elements of $\mathfrak Q^{-1}(Q)$. 
In Section~\ref{boxspecialsec} we show elements of the Box Conjecture for some partitions $Q$ having three parts.\par
This article is a shortening and revision of \cite{IKVZ}. We focus here on just the Table Theorem and Box Conjecture. We also flipped the B/C hooks of the table: this is useful for an article in preparation with M. Boij, in which we will study the equations of loci $\mathfrak Z_P\subset \mathcal N_B,$ where $B=J_Q$ with $Q=(u,u-r)$ and $P$ is a partition in the table $\mathcal T(Q)$. Here $\mathfrak Z_P$ is the Zariski closure in $\mathcal N_B$ of 
\begin{equation}\label{locidefeqn}
\{A\in \mathcal N_B\mid A \text { has Jordan type } P\}.
\end{equation}
We believe that the present article and its sequel yield a new approach to viewing the map $\mathfrak Q: P \to \mathfrak Q(P)$. While our methods are elementary, our results suggest that there might be interesting algebraic and geometric explanations and consequences.\par
\section{Preliminaries and Background.}\label{prelimbackgroundsec}
\subsection{Notation and Preliminaries.}\label{notationprelimsec}
We fix notation and summarize some concepts and results we will need. Let $P=(p_1,\ldots, p_s)$ be a \emph{partition} of the positive integer $n$ having $s$ parts. This means that $ p_1\ge \cdots \ge p_s>0$ and $p_1+p_2+\cdots +p_s=n$. We denote by $S_P$ the set of parts of $P$, i.e.
$S_P=\{p_1,p_2,\ldots ,p_s\}$. Note that $1\le |S_P|\le s$.
Recall that the Ferrers diagram of $P$ has rows whose lengths are the parts of $P$, which we arrange with the row $p_i$ above the row $p_{i+1}, i\in\{1,\ldots ,s-1\}$.  We denote by $P^\vee$ the conjugate partition to $P$:  the rows of the Ferrers diagram of $P^\vee$ are the columns of the Ferrers diagram of $P$. 
The \emph{regular partition} of $n$, denoted by  $[n]$ or $(n)$, is the only partition of $n$ with a single part.  Recall from Definition \ref{ARdef} that an almost rectangular partition $[n]^k$ 
of $n$ has $k$ parts whose maximum pairwise difference is zero or one.  We denote by $s^k$ the partition of $k\cdot s$ having $k$ parts equal to $s$.
Write $n=qk+{\sf r}$ with ${\sf r},q\in \mathbb N$ and $ 0\le {\sf r}< k$ and put  $d=k\cdot \lceil {\frac{n}{k}}\rceil -n$. Then 
\begin{equation*}
d=\begin{cases}k-{\sf r}&\text { if } {\sf r}\not=0\\
0 &\text {  if } {\sf r}=0,
\end{cases}
\end{equation*}
and we have that the almost rectangular partition $[n]^k$ satisfies
\begin{equation}\label{nkeq} [n]^k=\left( (q+1)^{\sf r},q^{k-{\sf r}}\right)=\left(\big\lceil {\frac{n}{k}}\big\rceil^{k-d},\big\lfloor \frac{n}{k}\big\rfloor^d\right).
\end{equation}
\begin{figure}
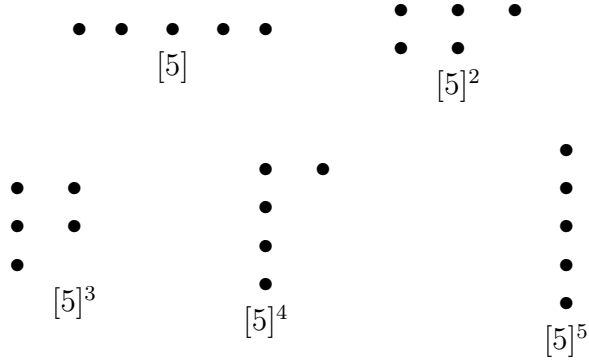

$\qquad\qquad\quad\qquad\qquad\qquad
\quad\begin{array}{ccccc}
\bullet&\bullet&\bullet&\bullet&\bullet\\
&&[5]&&
\end{array}\qquad\quad
\begin{array}{ccc}
\bullet&\bullet&\bullet\\\bullet&\bullet&\\
&[5]^2&
\end{array}
$
\vskip 0.3cm
$
\qquad\qquad\quad\qquad\qquad\quad\begin{array}{ccc}
\bullet&\bullet&\\
\bullet&\bullet&\\
\bullet&&\\
&{[5]^3}&
\end{array}\qquad\quad
\begin{array}{cc}
\bullet&\bullet\\
\bullet&\\
\bullet&\\
\bullet&\\
{[5]^4}&
\end{array}\qquad\qquad\qquad
\begin{array}{cc}
\bullet&\\
\bullet&\\
\bullet&\\
\bullet&\\
\bullet&\\
{[5]^5}&
\end{array}
$
\caption{The almost rectangular partitions of $5$.}\label{5slmostrectfig}
\end{figure}
See Figure \ref{5slmostrectfig} for the almost rectangular partitions of $5$.
\vskip 0.2cm
\noindent {\it Notation}: Given any partition $P$ of $n$ we denote by $J_P$ the unique Jordan matrix whose diagonal Jordan blocks have lengths $p_1,\ldots ,p_s$, arranged in descending order of lengths.  Given a nilpotent $n\times n$ matrix $A$ we denote by $P_A$ its \emph{Jordan type}; it is the partition giving  the  sizes of the blocks of the Jordan block matrix similar to $A$ (we write $J_{P_A}\sim A$).  Recall that the corank of $A$ is $(n-$ rank  $A)$, the dimension of the kernel of $A$. We take $A^0=I_n$, the $n\times n$ identity matrix. The following result is standard (see, e.g., \cite[Lemma 6.2.2]{CM}). 
\begin{lemma}\label{partsilem} The number of parts greater than or equal to $i$ in $P_A$ is the  difference
\begin{equation}
\text { corank } A^i- \text { corank }A^{i-1}.
 \end{equation}
\end{lemma}
\begin{lemma}\label{PBklem} Let $B= J_{(n)}$. Then  $P_{B^k}=[n]^k$. 
\end{lemma}
\begin{proof} Evidently $B^k$ has corank $k$. The number of parts of $P_A$ is the corank of $A$ so  $P_{B^k}$ has $k$ parts.  Let  $q=\lfloor \frac{n}{k}\rfloor$. Then 
$(B^k)^{(q+1)}=0$, so no part of $P_{B^k}$ is greater than $q+1$. Also, writing $n=kq+r$ with $0\le r<k$ we have $(B^k)^q$ has corank $kq$ and rank $r$: since no part is greater than $q+1$, this implies that there are exactly $r$ parts of $P_{B^k}$ equal to $q+1$. It follows that the remaining $k-r$ parts are each equal to $q$, so $P_{B^k}$ is almost rectangular and equal to $[n]^k$.  \end{proof}\par
This allows us to describe $\mathfrak Q^{-1}(Q)$ when $Q=(n)$ has a single part. 

\begin{corollary}  \label{cor:commute_with_principal}
If $A$ is a nilpotent matrix commuting with $J_{(n)}$ then $P_A=[n]^k$ for some $k$. Consequently, $\mathfrak Q^{-1}([n])$ is the set of almost rectangular partitions $\{[n]^k,1\le k\le n\}$. 
\end{corollary}
\begin{proof} Lemma \ref{PBklem} implies that if $B=J_{(n)}$ then $P_{B^k}=[n]^k$, so $\mathfrak Q([n]^k)=(n)$.  The matrices $A$ commuting with a regular nilpotent matrix $B$ are the polynomials $A=p(B)$ where $p\in {\sf k}[x]$ \cite[Theorem 2.8]{DGKO}. When $p=x^k\cdot p'$ where $ p'=a_ux^u+\cdots + a_0$ with $ a_0\not=0$ then $p'(B)$ is invertible, so $A=p(B)\sim B^k$ and $P_A=[n]^k$.
\end{proof}

\begin{comment}
 Arranging $\mathfrak Q^{-1}([n])$ in a linear table $\mathcal T([n])$ we have
\begin{equation}\label{almrecttableeq}
\begin{array}{c|cccc|}
\# \text { parts }&1&2&\cdots & n\\
\hline\hline
\mathfrak Q^{-1}(Q)=&\{[n]&[n]^2&\ldots &[n]^n=1^n\}
\end{array}
\end{equation}
\end{comment}
Loosely speaking, the main result in this paper is the generalization of Corollary~\ref{cor:commute_with_principal} to the case where $Q$ has two parts.
Recall that $r_P$ is the smallest number of almost rectangular partitions whose union is $P$.  
\begin{theorem}(R. Basili \cite{Bas})\label{rpthm}  The partition $\mathfrak Q(P)$ has $r_P$ parts. 
\end{theorem}
The following result is shown for $\cha {\sf k}=0$ in \cite[Theorem 2.1 and Example 2.5a]{Pan}, and for general infinite $\sf k$ in 
\cite[Theorem 1.12]{BI}. We say that a partition $P$ is \emph{stable} if $\mathfrak Q(P)=P$.
\begin{theorem}\label{stable1thm} A partition is  stable if and only if its parts differ pairwise by at least two.
\end{theorem}
A partition whose parts differ pairwise by at least two is termed ``super distinct'' or ``Rogers-Ramanujan'' in the literature on partitions. T. Ko\v{s}ir and P. Oblak showed that $\mathfrak Q(P)$ is stable.
\begin{theorem}(T. Ko\v{s}ir and P. Oblak) \cite{KO})\label{stablethm} Suppose that $\cha {\sf k}=0$ or $\cha {\sf k}=p>n$ and $\sf k$ is infinite. Then the partition $\mathfrak Q(P)$ has parts that differ pairwise by at least two.
\end{theorem}

\begin{remark}\label{proofstablerem}
In \cite{KO}, Theorem \ref{stablethm} is stated for algebraically closed fields. In this somewhat technical remark, we review the Theorem's proof and explain why this hypothesis can be relaxed; that is, why the Theorem holds for any infinite field ${\sf k}$ with $\cha {\sf k}=0$ or $\cha {\sf k}>n$.
\par The proof of Theorem \ref{stablethm} depends on showing that when $B=J_P$, the Jordan block matrix of partition $P$, and the matrix $A\in \mathcal N_B$ is generic, then the Artinian algebra $\mathcal A={\sf k}[A,B]$ is Gorenstein: T. Ko\v{s}ir and P. Oblak show that the action of $\mathcal A$ on the vector space $V$ has a cocyclic vector or, equivalently, the poset  $\mathcal D_P$ (cf. Section~\ref{posetmapsec} below) whose elements are a basis of $V$ has a ``sink''. Since $\mathcal A$, being a quotient of the local ring ${\sf k}\{x,y\},$ has height two, that $\mathcal A$ is Gorenstein implies by a result of F.H.S. Macaulay (\cite[\S 14]{Mac}) that $\mathcal A$ is a complete intersection: this means that $\mathcal A={\sf k}\{x,y\}/I$, the quotient of ${\sf k}\{x,y\}$ by an ideal $ I=(f,g)$ where $f,g$ are in general non-homogeneous elements of ${\sf k}\{x,y\}$. When $\cha {\sf k}=0$ or $\cha {\sf k}>n$, it follows that the Hilbert function of $\mathcal A$ is the conjugate of the partition $\mathfrak Q(P)$~\cite[Theorem~2.20]{BI}. \par  Although \cite[Theorem~2.20]{BI} and \cite[Theorem~2.16]{BI} upon which it depends are stated for $\sf k$ algebraically closed, only $\sf k$ infinite is needed, along with the condition on $\cha {\sf k}$ given in Theorem \ref{stablethm}. This is because \cite[Theorem 2.16]{BI} about pencils of matrices $A+\lambda B, \lambda\in \sf k$ uses that the ideal $I$ in the regular local ring ${\sf k}\{x,y\}$ defining the algebra $\mathcal A$ has a ``normal basis'' in the direction $x+\lambda y$, for an open dense set of $\lambda$ in the affine line $\mathbb A^1$. This is equivalent to $x+\lambda y$ being a ``strong Lefschetz element'' for $\mathcal A$ and occurs when $\lambda$ is not a root of a certain monic polynomial over $\sf k$.\par
The characterization of the Hilbert functions of (non-graded) Artinian CI algebras of height two by F.H.S. Macaulay \cite{Mac} now implies the property that $\mathfrak Q(P)$ has parts that differ pairwise by at least two. For a discussion see the original article \cite{KO} and as well \cite[Sections 2.4,2.5]{BIK}. 
\end{remark}

Denote the partition $P$ by $(\cdots  i^{n_i}\cdots )$ meaning it has $n_i$ parts of length $i$.  
An almost rectangular subpartition $P'=(a^{n_a},(a-1)^{n_{a-1}})$ of $P$ defines a so-called ``U-chain'' $C_a$, which is a certain union of chains in a partially ordered set $\mathcal D_P$ naturally associated to $P$.
In Section \ref{posetmapsec} we briefly recall the definition of the poset $\mathcal D_P$, which plays an important role in understanding the map $P\to \mathfrak Q(P)$, and the definition of a U-chain. For the main results of this paper (Theorems \ref{tablethm} and \ref{table2thm}) all we will need is that 
for an almost rectangular subpartition $(a^{n_a},(a-1)^{n_{a-1}})$ of $P$ the length (or number of elements) of the U-chain $C_a$ is
\begin{equation}\label{Caeq}
|C_a|=an_a+(a-1)n_{a-1}+2\sum_{i>a} n_i.
\end{equation}

\begin{theorem}(P. Oblak \cite{Obl1})\label{indexthm}  The largest part of $\mathfrak Q(P)$ is $\max\{|C_a|: a$ is a part of $P \}$.
\end{theorem}
This result was originally shown for $\cha {\sf k}=0$; the proof was subsequently seen to be valid over any infinite field $\sf k$: see \cite{BIK,IK}. It is the main tool we use to prove Theorems \ref{tablethm} and \ref{table2thm}.  Recall that when $r_P=2$ the partition $\mathfrak Q(P)$ has two parts, and therefore is completely determined by Theorem \ref{indexthm}.

\subsection{Background: the poset $\mathcal D_P$.}\label{posetmapsec}
We now recall the poset $\mathcal D_P$ associated to $P$.  This poset plays an important role in understanding the map $P\to \mathfrak Q(P)$.  For example, it is behind the proofs of Theorem  \ref{stablethm} of P.~Oblak and T. Ko\v{s}ir and Theorem~\ref{indexthm} of P. Oblak.  The main proofs of Section \ref{rp=2app} refer to  the U-chains in the poset. However, we note for those readers less interested in this background that the proofs there will use equation \eqref{Caeq} and Theorem \ref{indexthm} above and may be read independently of the Definition \ref{posetdef} of the poset $\mathcal D_P$. 
 For further discussion of $\mathcal D_P$ see \cite{BIK,IK,Kh1,Kh2,KO}.
  \subsubsection{The poset $\mathcal D_P$.}\label{posetsec}
   Let $P$ be a partition of $n$, and let $B=J_P$ acting on the vector space $V$.  
 The poset  $\mathcal D_P$ has $n$ vertices corresponding to a basis $\mathfrak B$ of $V$ and we will decompose $\mathfrak B$ into the union of bases for submodules $V_i$.
First we recall the basis $\mathfrak B$, which we will label by certain triples $({\mathfrak u,i,k})$. We write $n_i$ for the multiplicity of the part $i$ in $P$, so $P=(\ldots  , i^{n_i},\ldots )$. 
Following the notation of \cite[Section 2.1]{BIK} or \cite{IK} we have $V=\oplus_{i\in S_P} V_i$, where $V_i$ has a decomposition
\begin{equation}\label{Vsumeq}
V_i=\bigoplus_{k=1}^{n_i} V_{i,k}
\end{equation}
into cyclic $B$-modules $V_{i,k}$, each of length $i$. The subspace $V_{i,k}$ has a  cyclic vector that we name $(1,i,k)$ and $V_{i,k}$ has basis 
\begin{equation}\label{Vbasiseq}
 \{(\mathfrak u,i,k)=B^{\mathfrak u-1}(1,i,k):\, 1\le \mathfrak u\le i\}.
\end{equation}
So $V_{i,k}\cong {\sf k}[x]/x^i$ as a ${\sf k}[x]$-module through the action of $B$ \cite[Definition 2.3]{IK}.
 We denote by $\mathfrak B$ the concatenation of the above bases for $V_{i,k}$, and by $\langle A\cdot v\mid (\mathfrak u,i,k)\rangle$ the coefficient of $A\cdot v$ on the basis vector $(\mathfrak u,i,k)$. Fix $i$ and denote by $\mathfrak W_i$ the subset of $\mathcal B$ consisting of the cyclic vectors of $\{V_{i,k}:1\le k\le n_i\}$, that is\footnote{We thank a referee for pointing out that the order needs to be as given to be consistent with ``upper triangular'' in the definition of $\mathcal U_B$. See also Example \ref{Uex}.},
 \begin{equation}
{ \mathfrak W}_i=\{(1,i,n_i),\ldots ,(1,i,1)\}.
 \end{equation}
 Let $W_i$ be the span of $\mathfrak W_i$.   Denote by $\pi_i$ the projection from the centralizer $\mathcal C_B$ to $\Mat_{n_i}(\sf k)$ obtained by restricting $A\in \mathcal C_B$ to $W_i$ and then projecting to $W_i$. Let 
 \begin{equation}
 \pi: \mathcal C_B\to \prod_i \Mat_{n_i}(\sf k)
 \end{equation}
  be the product of the $\pi_i$. We define  a nilpotent subalgebra  $\mathcal U_B\subset \mathcal C_B$ as the set of all $A\in \mathcal C_B$ such that $\pi_i(A)$ is strictly upper triangular on $W_i$ for every $i$:
 \begin{equation}\label{UBeq}
 \mathcal U_B=\{ A\in \mathcal C_B\mid \text{ for every } i\in S_P: 1\le s\le s'\le n_i  \Rightarrow \langle A\cdot (1,i,s')\mid (1,i,s)\rangle=0 \}.
 \end{equation}
 The following proposition is well known (see \cite[Lemma 2.3]{Bas}, \cite[Thm. 3.5.2]{DrKi}, \cite[Theorem 6]{HW},\cite{TA}). Recall that $\mathcal N_B$ is the set of nilpotent elements of $\mathcal C_B$.
 \begin{proposition}\label{ubnbprop} Let $B=J_Q$ where $Q$  is a partition of $n$.  Then
 \begin{enumerate}[a.]
 \item  The map $\pi$ is the projection from $\mathcal C_B$ onto its semisimple part.
 \item $\mathcal U_B\subset \mathcal N_B$ and $\mathcal U_B$ is isomorphic to an affine space.
 \item When $P$ has no repeated parts, then $\mathcal U_B=\mathcal N_B$.
 \end{enumerate}
 \end{proposition}
 \begin{proof} Part (c) follows evidently from (a) and (b) when each $n_i=1$. The inclusion in part (b) follows from (a) when each $n_i=1$ since then $\pi\mid_{\mathcal U_B}=0$, and in general the inclusion follows from Lemma 2.3 in \cite{Bas}. Since it is an algebra, $\mathcal U_B$ is isomorphic to an affine space.
 \end{proof}
 \begin{definition}[Poset $\mathcal D_P$]\label{posetdef} The poset $\mathcal D_P$ has the set $\mathcal B$ of basis elements of $V$ as its underlying set. For $v,v'\in \mathfrak B$, we set $v<v'$ if there is an element 
 $A\in \mathcal U_B$ such that $\langle A\cdot v\mid v'\rangle\not=0$.
 \end{definition}
The \emph{diagram} Diag($\mathcal L$) of a poset $\mathcal L$ is a directed graph of which the vertices are the elements of $\mathcal L$ and with an arrow $v\to v'$ if $v'$ {covers} $v$ (here $v'$ \emph{covers} $v$ if  $ v< v'$ and there is no $ v''$ such that $ v<v''<v'$).
Recall that $S_P$ is the set of integers that are parts of $P$. For $i\in S_P$ we denote by $i^-$ the next smaller element of $S_P$ if it exists (that is, if $i$ is not the smallest part of $P$), and by $i^+$ the next larger element of $S_P$, if it exists.  For $P=(5,4,4,3,2,2)$ where $  S_P=\{5,4,3,2\}, 4^+=5$ and $3^-=2$.
\begin{definition}[Elementary Maps associated to $P$]\label{DPdef} 
\cite[Def. 2.9]{BIK}.
The maps $\beta_i,\alpha_i, e_{i,k}$ and $w_i$ defined below are zero on the elements of $\mathcal B$ that are not specifically listed. They are called the \emph{elementary maps} associated to $P=(p_1\ge p_2\ge \cdots \ge p_s)$. 
\begin{enumerate}[i.]
\item for $i\in S_P\backslash \{p_s\}$, $\beta_i$ maps the vertex $(\mathfrak u,i,n_i)$ to $(\mathfrak u,i^-,1)$, whenever $ 1\le \mathfrak u\le i^-$.
\item for $i\in S_P\backslash \{p_s\}$, $\alpha_{i}$ maps $(\mathfrak u,i^-,n_{i^-})$ to $ (\mathfrak u+i-i^-,i,1),$ whenever $ 1\le \mathfrak u\le i^-.$
\item  For $i\in S_P$ and $k\in \{1,2,\ldots ,n_i\}$ $e_{i,k}$ maps the vertex $(\mathfrak u,i,k)$ to $(\mathfrak u,i,k+1)$ whenever $ 1\le \mathfrak u\le i,1\le k<n_i$.
\item When $i\in S_P$ is isolated (i.e. when neither $i-1\in S_P$ nor $ i+1\in S_P$),   $w_i$ sends $(\mathfrak u,i,n_i)$ to 
$(\mathfrak u+1,i,1)$ whenever   $1\le \mathfrak u<i$.
\end{enumerate}
\end{definition}
\begin{lemma}\label{DPdefb} 
There is an edge $v\to v'$ in the diagram $\mathrm{Diag}(\mathcal D_P)$ if and only if there exists an elementary map $\gamma$ such that $\gamma (v)=v'$.  Also, the elementary maps generate the algebra $\mathcal U_B$.
\end{lemma}
\begin{proof}  That the elementary maps generate $\mathcal U_B$ is \cite[Corollary 2.12]{BIK}. This implies that there is an edge $v\to v'$ only if there is an elementary map $\gamma$ such that $\gamma (v)=v'$. Conversely, it is an easy case by case check that $v'$ covers $v$ when there is an elementary map $\gamma$ such that $\gamma (v)=v'$. \end{proof}
\Large
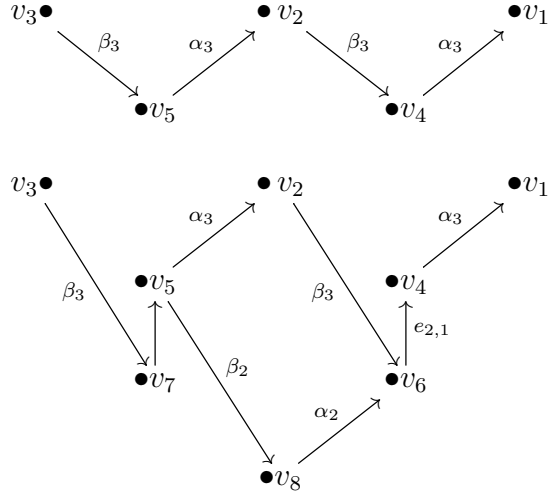
\begin{figure}[hbt]
\begin{equation*}
\begin{split}
\xymatrix{
{v_3\bullet}\ar[dr]^{\beta_3}&&{\bullet\,v_2}\ar[dr]^{\beta_3}&&{\bullet}v_1\\
&{\bullet}v_5\ar[ru]^{\alpha_3}&&{\bullet}v_4\ar[ru]^{\alpha_3}\\
}
\end{split}
\end{equation*}
\begin{equation*}
\begin{split}
\xymatrix{
{v_3\bullet}\ar[ddr]_{\beta_3}&&{\bullet\,v_2}\ar[ddr]_{\beta_3}&&{\bullet}v_1\\
&{\bullet}v_5\ar[ru]^{\alpha_3}\ar[rdd]^{\beta_2}&&{\bullet}v_4\ar[ru]^{\alpha_3}\\
&{\bullet}v_7\ar[u]&&{\bullet}v_6\ar[u]_{e_{2,1}}\\
&&{\bullet}v_8\ar[ru]^{\alpha_2}&&\\
}
\end{split}
\end{equation*}
\caption{$\mathrm{Diag}(\mathcal D_P)$ for $P=(3,2)$ and $   P=(3,2,2,1)$.}\label{smallDfig}
\end{figure}
\normalsize
\noindent
\begin{example} When $P=(3,2)$ and $B=J_P$ then the algebra $\mathcal U_B$ is generated by 
$\alpha_3,$ and $ \beta_3$, subject to the relations ${\alpha_3}^2={\beta_3}^2=(\beta_3\alpha_3)^2=0$. For $P'=(3,2,2,1)$ and $B'=J_{P'}$
the algebra  $\mathcal U_{B'}$ is generated by $\alpha_3,\alpha_2, \beta_3,\beta_2$ and $e_{2,1}$ (Figure \ref{smallDfig}).
\end{example}
\begin{definition}[Rows of $\mathcal D_P$]
A \emph{row of length $i$} of $\mathcal D_P$ is a subset of the form $\{(\mathfrak u,i,k)\in \mathcal D_P\mid 1\le \mathfrak u\le i\}$, where $i\in S_P$ and $k$ satisfying $1\le k\le n_i$ is fixed.
\end{definition}
  \begin{definition}\label{Udef} Let $a\in S_P$. The \emph{U-chain} $C_a$ of the poset $\mathcal D_P$ is comprised of three parts: \par
 i. the unique maximal chain through all  the vertices of $\mathcal D_P$ in rows of lengths $a$ and $ a-1$;\par
 ii. the maximal chain from the source vertex $(1,p_1,1)$ down to $(1,a,1)$;\par
 iii. the maximal chain from the vertex $(a,a,n_a)$ to the sink vertex $(p_1,p_1,n_{p_1})$ of $\mathcal D_P$.
 \end{definition}
By definition, the length $|C_a|$ is the number of vertices in the U-chain. It satisfies $|C_a|=an_a+(a-1)n_{a-1}+2\sum_{i>a} n_i$ (this is equation \eqref{Caeq}).
\begin{example}\label{Uex}
For the partition $P=(3,2,2,1)$ of Figure \ref{smallDfig}, the U-chain $C_2$ is\footnote{In writing a matrix for $A$ in the basis $\mathfrak B$ we number the basis vectors $v_1,\ldots ,v_n$ according to decreasing $i$, then decreasing $k$, then decreasing $\mathfrak u$, as in Figure \ref{smallDfig}; see also
 \cite{BI,BIK,IK}.} 
\begin{align*}
&v_3\to v_7\to v_5\to v_8\to v_6\to v_4\to v_1 \text{ or, in triples notation, }\\
&(1,3,1)\to (1,2,1)\to (1,2,2)\to (1,1,1)\to (2,2,1)\to (2,2,2)\to (3,3,1),
\end{align*}
 given by the chain of maps (right to left)   $\alpha_3\circ e_{21}\circ \alpha_2\circ \beta_2\circ e_{21}\circ \beta_3$. The U-chain $C_3$ is
 \begin{equation*}
(1,3,1)\to (1,2,1)\to (1,2,2)\to (2,3,1)\to (2,2,1)\to (2,2,2)\to (3,3,1)
\end{equation*}
given by  $\alpha_3\circ e_{21}\circ \beta_3\circ \alpha_3\circ e_{21}\circ \beta_3$. \par
\end{example}

\section{The table $\mathcal T(Q)$  for $\mathfrak Q^{-1}(Q)$ when $ Q=(u,u-r)$.}\label{rp=2app}
In this section we determine the tables $\mathcal T(Q)$ giving the complete set $\mathfrak Q^{-1}(Q)$ for all stable partitions $Q$ having two parts:  $Q=(u,u-r)$ with $u>r\ge 2$. Our main results, Theorem~\ref{tablethm}  specifying the table $\mathcal T(Q)\subset \mathfrak Q^{-1}(Q)$,   and Theorem \ref{table2thm} asserting completeness of the table are proved in Sections \ref{tablesec} and \ref{tablecompletenesssec}, respectively.

\subsection{Three subsets of $\mathfrak Q^{-1}(Q)$ and their intersections.}\label{threesubsetsec}

By Theorem \ref{rpthm}, $Q=\mathfrak Q(P)$ has two parts exactly when $P$ is the union of two almost rectangular  partitions, but $P$ is not almost rectangular. Hence there exist $a,b\in \mathbb N$ with $a-b\ge 2$ such that
\begin{equation}\label{Pabeq}P=\left(a^{n_a},(a-1)^{n_{a-1}}, b^{n_b}, (b-1)^{n_{b-1}}\right) \text {with $n_a>0$ and $n_b>0$}.
\end{equation} 
We can and will assume that $n_{b-1}=0$ if $b=1$.
As before, we have denoted by $n_i$ the number of parts of $P$ having length $i$.
\begin{definition}[Type A,B,C partitions in $\mathfrak Q^{-1}(Q)$]\label{casesdef}
Let $Q=(u,u-r)$ with $u,r\in \mathbb N, u>r\ge 2$ and let $P\in \mathfrak Q^{-1}(Q)$ satisfy \eqref{Pabeq}.\par
We say that $P$ is of \emph{type A} if $u=a\cdot n_a+(a-1)n_{a-1}$;\par
We say that $P$ is of \emph{ type B} if $u=2n_{a}+2n_{a-1}+bn_b+(b-1)n_{b-1}$, or if $b=a-2, n_{b-1}=0 $ and $u=2n_a+(a-1)n_{a-1}+bn_b$. \par
We say that $P$ is of \emph{type C} if $b=a-2$, if each of $n_a,n_{a-1},n_b,n_{b-1}$ is non-zero,  and $u=2n_{a}+(a-1)n_{a-1}+bn_b$. 
\end{definition}
\begin{remark} It is clear from Theorem \ref{indexthm} that every $P\in \mathfrak Q^{-1}((u,u-r))$ is of type A,B, or C. Note that a partition can have more than one type (Lemma \ref{Clem} and Remark \ref{tablerem}).
When $P$ has type A then the length of the U-chain $C_a$ through the upper almost rectangular subpartition of $P$ is the largest part of $Q$, and $u-r=bn_b+(b-1)n_{b-1}$.
When $P$ has type~B the length of the U-chain through the lower almost rectangular subpartition of $P$ is the largest part of $Q$ (when both $b=a-2$ and $n_{b-1}=0$ this almost rectangular lower subpartition is $((a-1)^{n_{a-1}},b^{n_b})$).
Then $u-r=(a-2)n_a$ if $b=a-2$ and $n_{b-1}=0$, and $u-r=(a-2)n_{a}+(a-3)n_{a-1}$ otherwise.
When $P$ has type C the middle almost rectangular U-chain $C_{a-1}$ is a longest U-chain. Then $u-r=(a-2)n_{a}+(b-1)n_{b-1}$.
\end{remark}
\begin{example} The partition $P=(5,4,4,3,3,2)$ is of type C since the middle U-chain of length $|C_4|=16$  is longest, as $|C_5|=13$ and $ |C_3|=14$.  The partition  $P=(5,5,4,3,2)$ is of type A: the longest U-chain is $C_5$. The partition
$P=(5,4,3,2,2,2)$ if of type B: the longest U-chain is the bottom chain $C_3$. 
\end{example}
We characterize below in Lemma \ref{Clem} partitions $P$ of type C that are not of type A or B; we use this result later in the proof of Lemma \ref{Ccompletelem} and Theorem~\ref{table2thm}.
The following is a consequence of equation \eqref{Caeq}.
\begin{lemma}\label{countlem} Let $P$ be a partition as in equation \eqref{Pabeq}.
The length of the top U-chain $U_{top}=C_a$ of $\mathcal D_P$ is
\begin{equation}\label{ctopeq}
|U_{top}|=n_a\cdot a+n_{a-1}\cdot (a-1),
\end{equation}
while the length of the bottom U-chain $U_{bottom}=C_b$ is
\begin{equation}\label{cboteq}
|U_{bottom}|=n_b\cdot b+n_{b-1}\cdot (b-1)+2(n_a+n_{a-1}).
\end{equation}
We have
\begin{equation}
|U_{top}|-|U_{bottom}|=n_a\cdot (a-2)+n_{a-1}\cdot (a-3)-n_b\cdot b-n_{b-1}\cdot (b-1).
\end{equation}
If $b=a-2$ and $n_{a-1}>0$, then the length of the middle U-chain $U_{middle}=C_{a-1}$ is
\begin{equation}\label{cmideq}
|U_{middle}|=n_{a-1}\cdot (a-1)+n_{a-2}\cdot (a-2)+2n_a,
\end{equation}
and we have
\begin{align}
|U_{middle}|-|U_{top}|&=(n_{a-2}-n_a)\cdot (a-2),\notag\\
|U_{middle}|-|U_{bottom}|&=(n_{a-1}-n_{a-3})\cdot (a-3).
\end{align}
Consequently, $P$ is of type C and not of type A or B if and only if $b=a-2$ and both
\begin{equation}
n_{a-1}>n_{a-3}>0, \, \text { and } \, n_{a-2}>n_a.
\end{equation}
\end{lemma}
\subsubsection{Classification of type C partitions.}
\begin{definition}\label{Cdef} Given the sequence  $C=({\sf c}_1,{\sf c}_2,s_1,s_2;a)$ of non-negative integers satisfying 
\begin{equation}\label{Ceq}
{\sf c}_1,{\sf c}_2\ge 1, a\ge 4,
\end{equation} we denote by $P_C$ the partition 
\begin{equation}\label{PCeq}
P_C=(a^{{\sf c}_1}, (a-1)^{{\sf c}_2+s_2}, (a-2)^{{\sf c}_1+s_1}, (a-3)^{{\sf c}_2}).
\end{equation}
\end{definition}\noindent
In the notation of Definition \ref{ARdef} we have
\begin{equation}\label{PC2eq}
P_C= \left( [({\sf c}_1+{\sf c}_2+s_2)a-({\sf c}_2+s_2)]^{{\sf c}_1+{\sf c}_2+s_2},[({\sf c}_1+{\sf c}_2+s_1)(a-2)-{\sf c}_2]^{{\sf c}_1+\sf{c}_2+s_1} \right),
\end{equation}
and $P_C$ is a partition of
\begin{align}
n&=a\cdot {\sf c}_1+(a-1)({\sf c}_2+s_2)+(a-2)({\sf c}_1+s_1)+(a-3){\sf c}_2\notag \\
&=(2a-2)\cdot {\sf c}_1+(2a-4)\cdot {\sf c}_2+(a-1)\cdot s_2+(a-2)\cdot s_1.\label{nCeq}
\end{align}
The number of parts $t(P_C)$ of $P_C$ satisfies $t(P_C)=2{\sf c}_1+2{\sf c}_2+s_1+s_2$.\par
The following Lemma is a direct consequence of Lemma \ref{countlem} and equation \eqref{PCeq}.
\begin{lemma}\label{Clem}
Let the sequence $C= ({\sf c}_1,{\sf c}_2,s_1,s_2;a)$ satisfy \eqref{Ceq}. Then $P_C$ is a type \emph{C} partition and
\begin{equation}\label{QPCeq}
\mathfrak Q(P_C)=\left(({\sf c}_2+s_2)(a-1)+({\sf c}_1+s_1)\cdot(a-2)+2{\sf c}_1,{\sf c}_1\cdot (a-2)+{\sf c}_2\cdot (a-3)\right).
\end{equation}
In other words, $P_C\in \mathfrak Q^{-1}((u,u-r))$ where
\begin{align}\label{PCeqn}
&u=({\sf c}_2+s_2)(a-1)+({\sf c}_1+s_1)\cdot(a-2)+2{\sf c}_1 \text { and }\notag\\
&r=s_1\cdot (a-2)+s_2\cdot (a-1)+2({\sf c}_1+{\sf c}_2).
\end{align}
Moreover, if $P$ is a partition of type C, then $P=P_C$ for some sequence $C=({\sf c}_1,{\sf c}_2,s_1,s_2;a)$ satisfying \eqref{Ceq}.
Here $P_C$ is also of type A if and only if $s_1=0$; and $P_C$ is also of type B if and only if $s_2=0$.
 The set
 \begin{equation}\label{completeCeq}
 \{P_C: C=({\sf c}_1,{\sf c}_2,s_1,s_2;a) \text { satisfying \eqref{Ceq} with $s_1\ge 1$ and $s_2\ge 1$}\}
 \end{equation}
  is the complete set of partitions that are of type C but not of type A or B.
\end{lemma}\noindent
The table $\mathcal T(Q), Q=(12,3)$ is the first to contain a type C partition $P_C, C=(1,1,1,1;4)$ that is not of type A or B; see also Example \ref{27,3ex} below where $Q=(27,3)$.
\begin{cor}\label{CCor} If $u\ge r+r^2/8$ then every partition in $\mathfrak Q^{-1}((u,u-r))$ has type A or has type~B.
\end{cor}
\begin{proof} Let $C$ and $P_C$ be as in equations \eqref{Ceq} and \eqref{PCeq}. By Lemma \ref{Clem},  both $s_1>0,s_2>0$ is needed for $P_C$ to not be of type A or B. Then 
\begin{equation*}r^2/8>\frac{s_1+s_2}{2}\cdot ({\sf c}_1+{\sf c}_2)(a-2)\ge ({\sf c}_1+{\sf c}_2)(a-2)>u-r=({\sf c}_1+{\sf c}_2)(a-2)-{\sf c}_2,
\end{equation*} 
where the first inequality follows from \eqref{PCeqn} and the last one uses that $c_2>0$.
\end{proof}
\begin{remark}\label{Crem} Although we will not use it, we note the following connection between type C partitions in the sets $\mathfrak Q^{-1}(Q)$ and $\mathfrak  Q^{-1}(Q')$ where $Q$ and $Q'$ are certain stable partitions of distinct integers $n$ and $n'$.
The formulas \eqref{PCeq} for $P_C$, \eqref{nCeq} for $n=|P_C|$, and \eqref{QPCeq} for $\mathfrak Q(P_C)$ are linear in the multiplicities $({\sf c}_1,{\sf c}_2,s_1,s_2)$. Also, increasing $a$ by $1$ to form $P'=P_{C'}=P_C+\underline{1}$ increases each part of $P_C$ by 1, so $|P_C'|= |P_C|+t(P_C)$, while the multiplicities stay the same. The same increase of $a$ by $1$ increases
$\mathfrak Q(P_C)=(u,u-r)$ by $\Delta Q=({\sf c}_1+{\sf c}_2+s_1+s_2, {\sf c}_1+{\sf c}_2)$ to form $Q+\Delta Q=Q'=(u',u'-r')$. It increases $r$ by $(s_1+s_2)$ and what we will call the \emph{key} ${\sf S}_{Q}=(r-1,u-r)$ of $Q$ to ${\sf S}_{Q'}={\sf S}_Q+\left(s_1+s_2,{\sf c}_1+{\sf c}_2\right)$
(see Definition \ref{keydef}). By setting $a=4$ we find the most ``basic'' partition $P_{C_0},C_0=({\sf c}_1,{\sf c}_2,s_1,s_2;4)$ of type C having given multiplicities $({\sf c}_1,{\sf c}_2,s_1,s_2)$. We have that $|P_C|=|P_{C_0}|+t(a-4)$, where $t=t(P_C)$.\par
 Finally, we note that it follows from \eqref{PCeqn} that the number of parts of $P_C$ satisfies
\begin{equation}\label{partsboundCeq}
t(P_C)\le \min \{ 2u/3, r\},
\end{equation}
for all type C partitions from Definition \ref{Cdef}.
We refer to this in Remark \ref{tablerem}.
\end{remark}

\subsection{The table $\mathcal T(Q)$ for $Q=(u,u-r)$.}\label{tablesec}
In this section we prove Theorem \ref{tablethm} which describes the $(r-1)\times (u-r)$ table $\mathcal T(Q)$ of elements in $\mathfrak Q^{-1}(Q)$.
\begin{definition}[Table invariants]\label{tablenotationdef} Let $Q=(u, u-r)$ with $u>r\geq 2$.  
If $r\ge 3$ and $1\leq t\le \min\{u-r, \lfloor\frac{r-1}{2}\rfloor\}$, define
$$\begin{array}{ll}
q_t&=\lceil\frac{u-r}{t}\rceil\\ \\
d_t&=t\cdot q_t-(u-r).
\end{array}$$
\medskip
We set $k_{0}=1$, and if $r\geq 3$ then for $1\leq t\le \min\{u-r, \lfloor\frac{r-1}{2}\rfloor\}$ define

\begin{equation}\label{kteq}\begin{array}{ll}
k_t&=\lceil \frac{u+1-t+d_t}{q_t+1} \rceil, \mbox{ and}\\ \\
c_t&=\left\{
\begin{array}{lll}
k_t&&\mbox{ if } d_t=0, \\ \\ 
\lceil \frac{u-2t+d_t}{q_t}\rceil&&\mbox{ if } d_t>0.
\end{array}\right.
\end{array}
\end{equation}
\end{definition}
\begin{remark}\label{3.10rem}
The invariants $k_t$ give the rows of the table $\mathcal T(Q)$ that start with a partition of type B or C, and the invariants $c_t$ determine the number of partitions of type C in the corresponding hook,  see Theorem \ref{tablethm}.\footnote{We recognize that the values of $k_t$ and $c_t$ are somewhat mysterious: they have to do with shifts in the ranks of certain powers of $P_{k,\ell}$ as the index $(k,\ell)$ changes.  We believe that a subsequent planned paper on equations of loci $\mathfrak Z_{P_{k,\ell}}$ (see equation \eqref{locidefeqn}) will shed some light.}\par
Since $0\leq \lceil\frac{u-r}{t}\rceil-\frac{u-r}{t}<1$, we have $0\leq d_t<t$.
Note that by  equation \eqref{nkeq}
$q_t$ and $d_t$ are defined in such a way that 
\begin{equation}\label{qtdteq}
[u-r]^{t}=((q_t)^{t-d_t}, (q_t-1)^{d_t}).
\end{equation}
Using the definition of $d_t=t\cdot q_t-(u-r)$, we can also write 
\begin{equation}\label{kalternative}
\begin{array}{ll}
k_t&=t+\lceil \frac{r-2t+1}{q_t+1} \rceil, \mbox{ and } \\ 
c_t&=\left\{
\begin{array}{lll}
k_t&&\mbox{ if } d_t=0, \\ \\ 
t+\lceil \frac{r-2t}{q_t}\rceil&&\mbox{ if } d_t>0.
\end{array}\right.

\end{array}
\end{equation}
\end{remark}
Example \ref{27,3ex} below shows these invariants for $Q=(27,3)$.
The following lemma gives some of the basic properties of the invariants of Definition \ref{tablenotationdef}. 

\begin{lemma}[Relations among the table invariants]\label{tableinvariantslem}
Assume that $Q=(u,u-r)$ with $u>r\geq 3$ and let $t_{\max}=\min\{u-r, \lfloor\frac{r-1}{2}\rfloor\}$. Then 

\begin{itemize}
\item[(a)] \qquad $k_{t_{\max}}=\left\{
\begin{array}{lll}
\frac{r}{2}&& \mbox{if }u-r> \lfloor\frac{r-1}{2}\rfloor \mbox{ and } r\mbox{ is even}; \\ \\ 
\lceil\frac{r+1}{2}\rceil&&\mbox{otherwise. }
\end{array} 
\right.$

\item[(b)] The sequence $\{k_1,k_2,\ldots ,k_{t_\max}\}$ is a nondecreasing sequence of positive integers satisfying $2\le k_t\le \lceil\frac{r+1}{2}\rceil$ for all $t\in\{1,2,\ldots , t_{\max}\}$.

\item[(c)] For all $t\in\{1,\ldots , t_{\max}\}$, $k_t\leq c_t \leq r-t.$ 
\end{itemize}
\end{lemma}

\begin{proof}
We begin with part (a).
First assume that $u-r\leq \lfloor \frac{r-1}{2}\rfloor$. Then $t_{\max}=u-r$. So $q_{t_{\max}}=1$, $d_{t_{\max}}=0$, and Definition \ref{tablenotationdef} gives $k_{t_{\max}}=\lceil\frac{r+1}{2}\rceil$.
\\

Now assume that $u-r> \lfloor \frac{r-1}{2}\rfloor$. Then $t_{\max}=\lfloor\frac{r-1}{2}\rfloor$ and by formula (\ref{kalternative}) we have 
$$k_{t_{\max}}=\Big\lfloor \frac{r-1}{2}\Big\rfloor+\Big\lceil\frac{r-2\left(\big\lfloor \frac{r-1}{2}\big\rfloor\right)+1}{q_{t_{\max}}+1}\Big\rceil.$$

Since $\frac{r-1}{2}-1<\lfloor\frac{r-1}{2}\rfloor\leq \frac{r-1}{2}$, we get $2\leq r-2(\lfloor \frac{r-1}{2}\rfloor)+1<4$. On the other hand, since by assumption $\lfloor \frac{r-1}{2}\rfloor <u-r$, we have $q_{t_{\max}}=\lceil\frac{u-r}{t_{\max}}\rceil \ge 2$.
Thus $\lceil\frac{r-2(\lfloor \frac{r-1}{2}\rfloor)+1}{q_{t_{\max}}+1}\rceil=1$ and therefore  by formula \eqref{kalternative} we obtain that
 \begin{equation}\label{tMaxcase2eq}
k_{t_{\max}}=\Big\lfloor \frac{r-1}{2}\Big\rfloor+1.
\end{equation}
To complete the proof of (a) we note that $\lfloor\frac{r-1}{2}\rfloor$ is equal to $\frac{r}{2}-1$ if $r$ is even, and is equal to $\frac{r-1}{2}$ if $r$ is odd.
\par
\noindent
We move to part (b).  By Definition \ref{tablenotationdef}, $q_1=u-r$ and $k_1=\lceil\frac{u}{u-r-1}\rceil$. So $k_1\geq 2$. By equation \eqref{kalternative}, for  $1< t \leq t_{\max}$ we have
$$\begin{array}{llll}
k_{t-1}&=t-1+\big\lceil \frac{r-2(t-1)+1}{q_{t-1}+1}\big\rceil && \\ \\
&\leq t-1+\big\lceil \frac{r-2t+1+1+q_t}{q_{t}+1} \big\rceil &&\mbox{ since } 1\leq q_t \leq q_{t-1}\\ \\
&=t+\big\lceil \frac{r-2t+1}{q_{t}+1}\big\rceil &&\\ \\
&=k_{t}.
\end{array}$$
\\ 
To complete the proof of part (b), it is enough to use part (a) to obtain that $k_{t_{\max}}\leq \lceil \frac{r+1}{2}\rceil$. 
\par
\noindent
Finally we prove part (c). Let $t\in\{1,2,\ldots , t_{\max}\}$. Using formula (\ref{kalternative}) and the definition of $c_t$, to show that $k_t\leq c_t$ it is enough to show that $\lceil \frac{r-2t+1}{q_t+1} \rceil \leq \lceil \frac{r-2t}{q_t} \rceil$. If $q_t\leq r-2t$, then $\frac{r-2t+1}{q_t+1}\leq \frac{r-2t}{q_t}$, and therefore the desired inequality holds. On the other hand, if $q_t > r-2t$, then both $\frac{r-2t}{q_t}$ and $\frac{r-2t+1}{q_t+1}$ are between 0 and 1, and therefore $\lceil \frac{r-2t+1}{q_t+1} \rceil =\lceil \frac{r-2t}{q_t} \rceil=1$. To complete the proof of part (c), using the inequality $\lceil \frac{r-2t+1}{q_t+1} \rceil \leq \lceil \frac{r-2t}{q_t} \rceil$, formula (\ref{kalternative}), and the definition of $c_t$, it is enough to show that $t+\lceil\frac{r-2t}{q_t}\rceil \leq r-t$ holds for all $t\in\{1,2,\ldots , t_{\max}\}$. This is obvious because by Definition \ref{tablenotationdef} the integer $q_t \geq 1$ and therefore  $\lceil\frac{r-2t}{q_t}\rceil \leq r-2t$. 
\end{proof}

The following is the first part of our main result. The reader may wish to read Corollary~\ref{hookcor} and Remark \ref{tablerem} along with the Theorem, to gain some intuition about the result. In particular, note that we privilege type A over type B over type C as  labels of partitions in $\mathfrak Q^{-1}(Q)$: thus, a set of indices $A_t$ correspond to type A partitions of the table $\mathcal T(Q)$ (Definition \ref{tabledef}) that may also be of type B or C;  a set $B_t$ corresponds to partitions of type B, but not of type A and a set $C_t$ corresponds to partitions in the table that are of type C but that are not of type A or B. To a certain extent, this favoring of type A over B over C is an arbitrary choice, but the choice does make the decomposition of the table more regular, in our view.
\begin{theorem}[Table Theorem, part I]\label{tablethm} Let $Q=(u, u-r)$ with $u>r\geq 2$. 
\begin{itemize}

\item[(a)] For every positive integer $t$ such that $1\le t\le \min\{u-r,\lfloor \frac{r-1}{2}\rfloor\}=t_{\max}$,  define the set 
\begin{equation*}
A_t=\{(k,\ell )\in \mathbb{N}\times\mathbb{N}\,|\, k_{t-1} \le k<k_{t}\mbox { and }t\le \ell\le u-r\}.
\end{equation*}

For $T=\min\{u-r,\lfloor \frac{r-1}{2}\rfloor\}+1=t_{\max}+1$ define the set 
\begin{equation*}
A_T=\{(k,\ell)\in \mathbb{N}\times\mathbb{N}\,|\, k=k_{t_\max}=r-T\text { and }T\le \ell \leq u-r\}.
\end{equation*}
Moreover 
\begin{equation}\label{ATnonempeq}
A_T\not=\emptyset \Leftrightarrow r \text{ is even and }u\ge \frac{3r}{2},
\end{equation}
and in that case $T=k_{t_\max}=\frac{r}{2}$

 Then for all $(k, \ell)\in A_t$, when $t\in \{1,2,\ldots ,T\}$ the partition $P_{k,\ell}=\left([u]^{k},[u-r]^{\ell}\right)$ is of type A and
satisfies $\mathfrak Q(P_{k, \ell})=(u,u-r)$.

\item[(b)] For every positive integer $t$ such that $1\le t\le\min\{u-r,\lfloor \frac{r-1}{2}\rfloor\}$, define the subset $C_t\subset  \mathbb{N}\times\mathbb{N}$ as 
$$C_t=\{(k,t)\,|\, k_t\leq k < c_t\}.$$ 
 Then for all $(k, \ell)\in C_t$, the partition $$P_{k,\ell}=\left([u-r+2t]^{t},[u-2t-d_t(q_t-1)]^{k-d_t}, (q_t-1)^{d_t})\right)$$
 is of type C but not of type A or B, and  satisfies $\mathfrak Q(P_{k, \ell})=(u,u-r)$.

\item[(c)] For every positive integer $t$ such that $1\le t\le\min\{u-r,\lfloor \frac{r-1}{2}\rfloor\}$, define the subset $B_t\subset \mathbb{N}\times\mathbb{N}$ as $$B_t=\{(k,t)\,|\, c_t\leq k \leq r-t\} \cup \{(r-t,\ell)\,|\,t< \ell \leq u-r\}.$$ 
Then for all $(k, \ell)\in B_t$, the partition $$P_{k,\ell}=\left([u-r+2t]^{t},[u-2t]^{k+\ell-t}\right)$$
 is of type B but not of type A and satisfies $\mathfrak Q(P_{k, \ell})=(u,u-r)$.

\item[(d)] Each pair $(k,\ell)\in \mathbb N\times \mathbb N$ with $1\le k\le r-1$ and $ 1\le \ell\le u-r$ belongs to one and only one set $A_t$, $B_t$ or $C_t$ defined above. In particular there are listed above $(r-1)(u-r)$ distinct partitions  $\{P_{k,\ell}\}$, each satisfying $\mathfrak Q(P_{k, \ell})=(u,u-r)$. Moreover, every partition $P_{k,\ell}$ has $k+\ell$ parts.

\end{itemize}

\end{theorem}
The proof of Theorem \ref{tablethm} starts on page \pageref{prooftablethm} after Remark \ref{tablerem}.

\begin{definition}\label{tabledef} 
\begin{enumerate}[(a)]
\item For $Q=(u,u-r)$ as in Theorem \ref{tablethm} we define the \emph{table} $\mathcal T(Q)$ as the array of partitions $\{P_{k,\ell} \mid 1\le k\le r-1, 1\le \ell\le u-r\}$ from Theorem \ref{tablethm}.
\item {[B/C hook]} For $1\le t\le\min\{u-r,\lfloor \frac{r-1}{2}\rfloor\}$ the set $\{ P_{k,\ell}\mid (k,\ell)\in B_t\cup C_t\}$ is called the \emph{$t$-th B/C hook of $\mathcal T(Q)$}. 
\item {[A row]} For a pair $(t,k)$ consisting of a positive integer $t$ such that $1\le t\leq \min\{u-r,\lfloor \frac{r-1}{2}\rfloor\}$ and $k$ satisfying $k_{t-1}\le k<k_t$, and also for $t=t_{max}+1$ and each $k$ satisfying $ k_{t-1}\le k\leq r-t$ the set of partitions $\{P_{k,\ell} \mid  t\le \ell\le u-r\}$ is called the \emph{$(t,k)$-th  A row} of $\mathcal T(Q)$, or, more simply, the $k$-th A row of $\mathcal T(Q)$.  \end{enumerate}
\end{definition}
 Evidently, when $1\le t\le t_\max$ the set $A_t$ of Theorem \ref{tablethm}(a) is the union of indices $(k,\ell)$ from all the $k$-th A rows of $\mathcal T(Q)$ for $k_{t-1}\le k<k_t$;  for $T=t_\max+1$ when the set $A_T$ is nonempty, it comprises a single A row with $k=r-T$. 

\begin{corollary}\label{hookcor} The A rows and B/C hooks form a decomposition of the table $\mathcal T(Q)$.
\end{corollary}
The proof of Corollary \ref{hookcor} is given on page \pageref{proof3.15}.
\begin{remark}[Table decomposition into  A rows and B/C hooks]\label{tablerem} 
By Theorem \ref{tablethm} and Corollary \ref{hookcor} the $(r-1)\times (u-r)$ table $\mathcal T(Q)$ is decomposed into (disjoint) A rows and B/C hooks, each ending in the rightmost column of $\mathcal T(Q)$. The $t$-th B/C hook begins at $P_{k_t,t}$, $t\in \{1,\ldots ,t_\max\}$, has a lower left corner at $P_{r-t,t}$ and moves horizontally right to $P_{r-t,u-r}$. The top $k_1-1$ rows of the table are comprised of type A partitions. Each subsequent $(t,k)$-th A row or partial row with $t\ge 2$ begins to the right of the $(t-1)$-st B/C hook, and above the $t$-th B/C hook ($k_{t-1}\le k<k_{t})$ for some $t$, or to the right of the last B/C hook ($k\ge k_{t_\max}$). The $(t,k)$-th (or $k$-th)  A row begins at $P_{k,t}$ and ends at $P_{k,u-r}$: this leaves exactly the $t-1$ spaces at the start of the $k$ row of $\mathcal T(Q)$ for the column portion of the previous $t-1$ B/C hooks, each beginning at  $P_{k_{t'},t'}$ for $t'$ satisfying $1\le t'\le t-1$. In other words,
these rows and hooks exactly fit together to 
form the rectangular table $\mathcal T(Q)$.  Note that the last B/C hook will be entirely horizontal if   $k_{t_\max}=r-t_\max$, and it will be vertical if $t_{max}=u-r$.
\par Sometimes a partition has more than one type: for example $P=(a,a-1,a-2,a-3)$ with $\mathfrak Q(P)=(2a-1,2a-5)$ is of type A,B, and C.  However, by Definition \ref{tabledef}  a B/C hook consists of partitions that are of type B or C, but not of type A. Finally, each type C entry $P_{k,\ell}$ is preceded in its  B/C hook only by other type C entries, and by \eqref{partsboundCeq} they can occur only when $k+\ell\le \min \{2u/3, r\}$.
\end{remark}
\begin{proof}[Proof of Theorem \ref{tablethm}]\label{prooftablethm}\par
\noindent{\bf Part (a)} \par
{\bf Case 1}. Let $1\leq t \le\min\{u-r, \lfloor \frac{r-1}{2}\rfloor\}$ and suppose that $(k,\ell)\in A_t$. Then $k_{t-1} \leq  k <k_{t}$ and $ t\le\ell \le u-r$. 

Since $k<k_t$, by the definition of $k_{t}$ we have 
\begin{equation}\label{k-t} k \leq k_t-1 =\Big \lceil \frac{u+1-t+d_t}{q_t+1} \Big\rceil-1 \leq \frac{u-t+d_t}{q_t+1}.\end{equation}

Thus  $ \frac{u}{k}  \geq \frac{u}{u-t+d_t}(q_t+1).$ Since $d_t< t$, we get $ \frac{u}{k}  > q_t+1.$  In particular we have
 $$ \Big\lfloor \frac{u}{k}\Big\rfloor  \geq q_t+1 \mbox{ and }\Big\lceil \frac{u}{k}\Big\rceil  \geq q_t+2.$$ 

On the other hand, by equation \eqref{qtdteq} we have $[u-r]^{t}=\Big(q_t^{t-d_t}, (q_t-1)^{d_t}\Big).$
\\

If $\lfloor \frac{u}{k} \rfloor \geq q_t+2$, then the two longest U-chains poset of $P_{k,\ell}$ are $U_{top}=C_{\lceil \frac{u}{k}\rceil}$ and $U_{bottom} =C_{\lceil \frac{u-r}{\ell}\rceil}$, of lengths

$$\begin{array}{rlll}
|U_{top}|&= u,
\\ \\
|U_{bottom}|&= u-r+2k\\
&\leq u-r+2(k_t-1)  &&(\mbox{by the first inequality of \eqref{k-t}})\\
&\leq u-r+2(\lceil\frac{r+1}{2}\rceil-1) &\hspace{.1 in}&(\mbox{by Lemma }\ref{tableinvariantslem}b) \\
&\leq u.
\end{array}$$

So $\mathfrak Q(P_{k, \ell})=(u, u-r)$ in this case and $P_{k,\ell}$ is of type A.
\\
Now assume that $\lfloor \frac{u}{k}\rfloor = q_t+1$.  We can then write 
$$P_{k,t}=\Big((q_t+2)^{m}, (q_t+1)^{k-m}, q_t^{t-d_t}, (q_t-1)^{d_t}\Big)$$ for some $m \in \{1, \cdots ,k\}$.
\\

By (\ref{k-t}), $k(q_t+1)+(t-d_t)\leq u$. Thus $[u]^{k}$ has at least  $t-d_t$ parts of size $q_t+2$. So $m \geq t-d_t>0$.

\medskip

Thus, the three longest U-chains in the poset of $P_{k,t}$ have the following lengths:
$$
\begin{array}{rlll}
|U_{top}|&=|C_{q_t+2}|=u, \\ \\
|U_{middle}|&= |C_{q_t+1}|=u+(u-r)-(d_t(q_t-1)+m\,q_t)\\
&\leq u+(u-r)-(d_t(q_t-1)+(t-d_t)q_t)\\ 
&=u+u-r-(u-r)=u,&\mbox{ (By definition of } q_t\mbox{)}
\\ \\
|U_{bottom}|&=|C_{q_t}|=u-r+2k\\
&\leq u-r+2(k_t-1) &(\mbox{By the first inequality in \eqref{k-t}})\\
&\leq u-r+2(\lceil\frac{r+1}{2}\rceil-1) \hspace{.1 in}&(\mbox{By Lemma }\ref{tableinvariantslem}b) \\
&\leq u.
\end{array}$$

So $\mathfrak Q(P_{k,t})=(u, u-r)$ in this case as well and $P_{k,t}$ is of type A. For $\ell>t$ the size $|C_{q_t+1}|$ for the poset of $P_{k,\ell}$ is strictly smaller than that for $P_{k,t}$, while the lengths  $|U_{top}|$ and $|U_{bottom}|=|C_{\lceil \frac{u-r}{\ell}\rceil}|$ are the same as for $P_{k,t}$. This implies that for $(k,\ell)\in A_t$ the partition $P_{k,\ell}$ is of type A, and $\mathfrak Q(P_{k,\ell})=Q$,  as claimed. This completes the proof of part (a) Case 1.
\medskip
\par
{\bf Case 2}.
Let $t=\min\{u-r, \lfloor \frac{r-1}{2}\rfloor\}+1=T$.  We first prove \eqref{ATnonempeq}.
We assume that $r$ is even and $u\ge \frac{3r}{2}$. Then $u-r>\lfloor \frac{r-1}{2}\rfloor =\frac{r}{2}-1$ and therefore $T=\frac{r}{2}$ and $k_{t_\max}=\frac{r}{2}$ by 
Lemma~\ref{tableinvariantslem}(a), and clearly $A_T\not=\emptyset$.\par
We turn to `$\Rightarrow$'. Assume that $A_T\not=\emptyset$. Then $T\not=u-r+1$ since otherwise there would be no $\ell$ satisfying $ T\le \ell\le u-r$. Consequently $T=\lfloor \frac{r-1}{2}\rfloor+1<u-r+1.$  Let $(k,\ell)\in A_T$. Then $k_{t_\max}=k=r-T$. We use Lemma~\ref{tableinvariantslem}(a). If $k_{t_\max}=\lceil \frac{r+1}{2}\rceil$ then we obtain the contradiction
\begin{equation*}
\big\lceil \frac{r+1}{2}\big\rceil=k_{t_\max}=r-T=r-\left(\big\lfloor \frac{r-1}{2}\big\rfloor+1\right)=r-\big\lfloor\frac{r+1}{2}\big\rfloor\le \frac{r}{2}.
\end{equation*}
Therefore $r$ is even and $k_{t_\max}=\frac{r}{2}$. This proves \eqref{ATnonempeq} and the claim that when $A_T\not=\emptyset$ then $T=k_{t_\max}=\frac{r}{2}$.
For the remainder of the proof of part (a) we assume that $r$ is even, $u\ge \frac{3r}{2}$, and $k=k_{t_\max}=T=\frac{r}{2}\le \ell\le u-r$. Then

\begin{equation*}
\begin{array}{l}
\displaystyle{\frac{u}{k}}=
\displaystyle{\frac{u}{\frac{r}{2}}}=\displaystyle{\frac{2u}{r}}, \,\text {    and }\\\\
\displaystyle{\frac{u-r}{\ell}}\leq
 \displaystyle{\frac{u-r}{\frac{r}{2}}}=\displaystyle{\frac{2u}{r}}-2.
\end{array}
\end{equation*}

Therefore $ \lceil \frac{u-r}{\ell}\rceil\leq \lfloor\frac{u}{k}\rfloor-1.$
\\

If $ \lceil \frac{u-r}{\ell}\rceil\leq \lfloor\frac{u}{k}\rfloor-2$ then the longest U-chains in the poset of $P_{k, \ell}$ have the following lengths:

$$\begin{array}{rlll}
|U_{top}|&= u,
\\ \\
|U_{bottom}|&= u-r+2k\\
&= u-r+2(\frac{r}{2}) && \\
&= u.
\end{array}$$

So $\mathfrak Q(P_{k, \ell})=(u, u-r)$ in this case and $P_{k,\ell}$ is of type A (and of type B).
\\

Now assume that $ \lceil \frac{u-r}{\ell}\rceil = \lfloor\frac{u}{k}\rfloor-1$. Let $q=\lceil \frac{u-r}{\ell}\rceil$. Then there exist nonnegative integers $n_1$ and $n_{0}$ such that $$[u]^k=\Big((q+2)^{n_1}, (q+1)^{k-n_1}\Big) \mbox{ and } [u-r]^\ell =\Big(q^{n_{0}}, (q-1)^{\ell-n_{0}}\Big).$$ 

In particular $u=(q+1)k+n_1$ and $u-r=(q-1) \ell+n_{0}$. Thus 

\begin{equation}\label{righthalf}
\begin{array}{ll}
n_1-n_{0}&=u-(q+1)k-[u-r-(q-1)\ell]\\ \\
&=(q-1)(\ell-k)-2k+r\\ \\
&\geq(q-1)(\frac{r}{2}-\frac{r}{2})-2\frac{r}{2}+r \text { (since $k=k_{t_\max}=T=\frac{r}{2}\le \ell$})\\\\
&= 0.
\end{array}
\end{equation}

Thus, the longest three U-chains in the poset of $P_{k,\ell}$ have the following lengths:
$$
\begin{array}{rlll}
|U_{top}|&=|C_{q+2}|=u, \\ \\
|U_{middle}|&= |C_{q+1}|=u+(u-r)-((\ell-n_{0})(q-1)+n_1\,q)\\
&=u+(u-r)-((\ell-n_{0})(q-1)+n_{0}\,q+(n_1-n_0)q)\\
&=u+(u-r)-(u-r+(n_1-n_0)q)\text { (by \eqref{righthalf})}\\
&\leq u+(u-r)-(u-r)=u,
\\ \\
|U_{bottom}|&=|C_{q}|=u-r+2k\\
&=u-r+2(\frac{r}{2})=u.

\end{array}$$
So $\mathfrak Q(P_{k,\ell})=(u, u-r)$ in this case as well and $P_{k,\ell}$ is of type A (and of type B and possibly also of type C).
\vskip 0.2cm

\noindent{\bf Part (b).} Let $1\leq t \le \min\{u-r, \lfloor\frac{r-1}{2}\rfloor\}$ and let $(k,\ell)\in C_t$. Then $\ell=t$ and $k_t\leq k < c_t$.

If $c_t\leq  k_t$ then $C_t$ is empty and therefore there is nothing to prove in this case. We assume that $c_t> k_t$. This in particular means $d_t>0$ by \eqref{kteq}. We also note that by definition $q_t>1$ unless $t=u-r$ in which case $d_t=0$. So we also have $q_t>1$.

Since $k_t \leq k$, using the definition of $k_t$ we get
 \begin{equation}\label{31biseq}
 k-d_t \geq \Big\lceil \frac{u+1-t+d_t}{q_t+1}\Big\rceil-d_t  \geq \frac{u+1-t+d_t}{q_t+1}-d_t=\frac{u+1-t-d_tq_t}{q_t+1} .
 \end{equation}
Since by Remark \ref{3.10rem} we have $d_t<t$, we also get
 \begin{equation}\label{3.22biseq}
 u-2t-d_t(q_t-1)=(u+1-t-d_tq_t)-(t+1-d_t)<(u+1-t-d_tq_t).
 \end{equation} 
Moreover, $u-2t-d_t(q_t-1)\ge 0$ and therefore $u+1-t-d_tq_t>0$ and $k-d_t>0$. 
Indeed $u-2t-d_t(q_t-1)>u-2t-t(q_t-1)=u-t-(d_t+(u-r))>u-2t-u+r=r-2t>0$.
Thus,
\begin{equation*}
\frac{u-2t-d_t(q_t-1)}{k-d_t} \leq \frac{u-2t-d_t(q_t-1)}{u+1-t-d_tq_t}\, (q_t+1).
\end{equation*} 
This implies by \eqref{3.22biseq}
\begin{equation}\label{abceq1}
\frac{u-2t-d_t(q_t-1)}{k-d_t} <\frac{(u+1-t-d_tq_t)}{u+1-t-d_tq_t}\, (q_t+1)=q_t+1.
\end{equation}

On the other hand, since $k<c_t$, using the definition of $c_t$ we get  
$$k-d_t \leq c_t-1-d_t =\Big\lceil \frac{u-2t+d_t}{q_t}\Big\rceil -1 -d_t \leq \frac{u-2t+d_t+q_t-1}{q_t}-1-d_t=\frac{u-2t-d_t(q_t-1)-1}{q_t}.$$

Thus
\begin{equation}\label{xyzeq1}
\frac{u-2t-d_t(q_t-1)}{k-d_t}\geq \frac{u-2t-d_t(q_t-1)}{u-2t-d_t(q_t-1)-1}q_t> q_t.
\end{equation}

It follows from \eqref{abceq1} and \eqref{xyzeq1} that we can write
 $$[u-2t-d_t(q_t-1)]^{k-d_t}=(\, (q_t+1)^{n_1}, q_t^{n_0}\,)$$ 
 with $
n_1\geq 1$. Moreover, we have by \eqref{31biseq} 
$$(k-d_t)(q_t+1)\geq u+1-t-d_tq_t=u-2t-d_t(q_t-1)+(t+1-d_t),$$ we get 
\begin{equation}\label{n0eq}
n_0 \geq t+1-d_t,
\end{equation}
and $n_0+n_1=k-d_t$ is positive.

Thus, using that $[u-r+2t]^{t}=((q_t+2)^{t-d_t},(q_t+1)^{d_t})$, as follows from \eqref{qtdteq}, we have
\begin{equation}\label{tableCeq}
P_{k,\ell}=((q_t+2)^{t-d_t},(q_t+1)^{d_t+n_1},q_t^{n_0}, (q_t-1)^{d_t}).
\end{equation}
 Therefore the three longest U-chains in the poset of $P_{k,\ell}$ have the following lengths:

$$
\begin{array}{rll}
|U_{top}|&=|C_{q_t+2}|\\&=(u+u-r)-(d_t(q_t-1)+n_0q_t)\\
&=u+u-r-d_t(q_t-1)-n_0q_t\\
&\leq (u+u-r)-d_t(q_t-1)-(t+1-d_t)q_t \text{ by \eqref{n0eq} }\\
&=u+u-r-d_tq_t-(-d_t+tq_t)-q_t+d_tq_t\\
&=u-q_t \qquad\qquad (\mbox{by the definition of } q_t), \\
&<u. \quad\qquad\qquad \qquad(\mbox {since } q_t> 1).
\\ \\

|U_{middle}|&=|C_{q_t+1}|\\&= (u+u-r)-\left[(t-d_t)(q_t+2)+d_t(q_t-1)\right]+2(t-d_t)\\
&=(u+u-r)-\left[tq_t-d_t\right]\\
&=u. \\ \\

|U_{bottom}|&=|C_{q_t}|\\&= (u+u-r)-[(t-d_t)q_t+(d_t+n_1)(q_t-1)]\\\\
		    &=(u+u-r)-[(t+n_1)q_t-(d_t+n_1)]\\\\
		    &=(u+u-r)-[tq_t-d_t+n_1(q_t-1)]\\\\
		    &=u-n_1(q_t-1) \qquad\qquad (\mbox{by the definition of }q_t)\\\\
		    &<u. \qquad\qquad(\mbox{since } q_t>1 \mbox{ and } n_1\geq 1).

\end{array}$$
So $\mathfrak Q(P_{k,\ell})=(u,u-r),$ as desired, and $P_{k,\ell}$ is of type C but not of type A or B.
\vspace{.3 in}

\noindent{\bf Part (c).} Let $1\leq t \le \min\{u-r, \lfloor\frac{r-1}{2}\rfloor\}$ and let $(k,\ell)\in B_t$. Then either $\ell=t$ and $c_t \leq k \le r-t$, or $k=r-t$ and $t<\ell\leq u-r$.
Recall $P_{k,\ell}=\left([u-r+2t]^{t},[u-2t]^{k+\ell-t}\right)$. \\
\par

{\bf Case 1.} Let $d_t=0$.

In this case $k_t=\displaystyle{\Big\lceil \frac{u+1-t}{q_t+1}\Big\rceil }$ and $[u-r+2t]^{t}=\Big((q_t+2)^{t}\Big).$ 
\\

On the other hand the assumption $d_t=0$ implies $c_t=k_t$, by definition of $c_t$. Since $k \geq c_t$ and $\ell \geq t$, we get 
$$
k+\ell-t\geq k_t=\Big\lceil \frac{u+1-t}{q_t+1}\Big\rceil \geq \frac{u+1-t}{q_t+1}. 
$$
Thus 
\begin{equation}\label{partc}
\begin{array}{lll}
\Big(k+\ell-t\Big)(q_t+1)&\geq u+1-t\\
&=u-2t+(t+1)\\
&>u-2t.&
\end{array}
\end{equation}

Therefore  $\lceil \frac{u-2t}{k+\ell-t}\rceil\leq q_t+1$.
\\

\begin{quote}
{\bf Case 1.1.} If $\lceil \frac{u-2t}{k+\ell-t} \rceil < q_t+1$, then the largest part of the partition \break $[u-2t]^{k+\ell-t}$ is at most $q_t$, and therefore it is not adjacent to the parts of $[u-r+2t]^{t}$. Thus the lengths of the two longest U-chains in the poset of $P_{k,\ell}$ are as follows.
$$
\begin{array}{rlll}
|U_{top}|&=u-r+2t\\
&\leq u-r+2\lfloor\frac{r-1}{2}\rfloor\\
&< u.
\end{array}
$$

\begin{align*}
|U_{bottom}|&=2t+(u-2t)\\
&=u.
\end{align*}

So $\mathfrak Q(P_{k,\ell})=(u,u-r)$ and $P_{k,\ell}$ is of type B but not of type A.
\end{quote}

\begin{quote}
{\bf Case 1.2.} If $\lceil \frac{u-2t}{k+\ell-t} \rceil = q_t+1$, then formula (\ref{partc}), thanks to \eqref{nkeq},  implies that the partition \break$[u-2t]^{k+\ell-t}$ has at least $t+1$ parts of size $q_t$. Thus we can write $$[u-2t]^{k+\ell-t}=((q_t+1)^{k+\ell-t-n_0},q_t^{n_0}),$$ with $n_0\geq t+1$. Therefore $P_{k,\ell}=(\,(q_t+2)^{t}, (q_t+1)^{k+\ell-t-n_0},q_t^{n_0}\, )$ and the lengths of the longest two U-chains in the poset of $P_{k,\ell}$ are as follows.
$$
\begin{array}{rlll}
|U_{top}|&=|C_{q_t+2}|\\
&=(u+u-r)-n_0q_t\\
&\leq (u+u-r)-(t+1)q_t\\
&=u+u-r-(u-r)-q_t &\quad (\mbox{since } u-r=tq_t)\\ 
&<u. &\quad (\mbox{since } q_t>0). \\ \\

|U_{bottom}|&=|C_{q_t+1}|\\
&=(u+u-r)-tq_t\\
&=u+u-r-(u-r) &\quad (\mbox{since } u-r=tq_t)\\
&=u. 
\end{array}$$
So $\mathfrak Q(P_{k,\ell})=(u,u-r)$ and $P_{k,\ell}$ is of type B but not of type A.
\end{quote}
This completes the proof of (c) in Case 1.
\bigskip

{\bf Case 2.} Let $d_t>0$.

In this case $[u-r+2t]^{t}=((q_t+2)^{t-d_t}, (q_t+1)^{d_t}).$

Since by assumption $k\geq c_t$, and $\ell \geq t$, by the definition of $c_t$ we also have
  $$k+\ell-t\geq c_t=\Big\lceil \frac{u-2t+d_t}{q_t}\Big\rceil \geq \frac{u-2t+d_t}{q_t}.$$ 

Therefore \begin{equation}\label{case2k_t}(k+\ell-t)q_t\geq u-2t+d_t>u-2t.\end{equation}

Thus $\lceil \frac{u-2t}{k+\ell-t}\rceil \leq q_t$. 

\begin{quote}
{\bf Case 2.1.} If $\lceil \frac{u-2t}{k-t+\ell}\rceil < q_t$ then the lengths of the longest two U-chains in the poset of $P_{k,\ell}$ are as follows.
$$
\begin{array}{rlll}
|U_{top}|&=u-r+2t\\
&\leq u-r+2\lfloor\frac{r-1}{2}\rfloor\\
&< u.
 \\ \\
|U_{bottom}|&=2t+[u-2t]\\
&=u. 
\end{array}$$

So $\mathfrak Q(P_{k,\ell})=(u,u-r)$ and $P_{k,\ell}$ is of type B but not of type A.

\end{quote}

\begin{quote}
{\bf Case 2.2.} If $\lceil \frac{u-2t}{k-t+\ell}\rceil = q_t$, then by (\ref{case2k_t}) the partition $[u-2t]^{k-t+\ell}$ has at least $d_t$ parts of size $q_t-1$. So we can write $$[u-2t]^{k+\ell-t}=(q_t^{k+\ell-t-n_{-1}},(q_t-1)^{n_{-1}}),$$ with $n_{-1}\geq d_t$. Thus 
$$P_{k,\ell}=((q_t+2)^{t-d_t}, (q_t+1)^{d_t}, q_t^{k+\ell-t-n_{-1}},(q_t-1)^{n_{-1}})$$ and the lengths of the longest three U-chains in the poset of $P_{k,\ell}$ are as follows.

$$
\begin{array}{rlll}
|U_{top}|&=|C_{q_t+2}|\\
&=u-r+2t\\
&\leq u-r+2\lfloor\frac{r-1}{2}\rfloor\\
&<u. \\ \\

|U_{middle}|&=|C_{q_t+1}|\\
&=(u+u-r)-[(t-d_t)q_t+n_{-1}(q_t-1)]\\
&\leq u+u-r-[(t-d_t)q_t+d_t(q_t-1)] \text { (by the definition of $q_t$ and $d_t$) }\\
&= u. \\ \\

|U_{bottom}|&=|C_{q_t}|\\
&=2t+[u-2t]\\
&=u. 
\end{array}$$

So $\mathfrak Q(P_{k,\ell})=(u,u-r)$ and $P_{k,\ell}$ is of type B but not of type A. (It may also be of type C.)
\end{quote}

This completes the proof in Case 2 and therefore the proof of part (c).
\vspace{.3 in}

\noindent{\bf Part (d)}. 
 It is easy to check that by construction each partition $P_{k,\ell}$ in parts (a),(b),(c) of the Theorem has $k+\ell$ parts. What remains is to show
\begin{enumerate}[(1)]
\item every $(k,\ell)$ with $1\le k\le r-1$ and $1\le \ell \le u-r$ belongs to exactly one of the sets $A_t,B_t,C_t$, and
\item  all partitions $P_{k,\ell}$ are distinct, that is, if $(k,\ell)\not=(k',\ell')$, then $ P_{k,\ell}\not=P_{k',\ell'}$.
\end{enumerate}
\noindent
We begin with the proof of assertion (1).  Let $(k,\ell) \in \{1,2,\ldots ,r-1\}\times \{1,2,\ldots, u-r\}$ and let $T=\min \{u-r,\lfloor {\frac{r-1}{2}}\rfloor\}+1=t_\max+1. $ 

Then either $1\leq \ell < T$ or $\ell \geq T$.
\vskip 0.2cm
Assume first that $1\leq \ell < T$. If $k<k_{\ell}$, then $(k, \ell) \in A_{\ell'}$, where $\ell'$ is the smallest positive integer such that $k<k_{\ell'}$.  If $k_{\ell-1}\le k<k_{\ell}$ then $(k, \ell) \in A_{\ell}$. If $k_{\ell}\leq k < c_{\ell}$ then $(k, \ell) \in C_{\ell}$. If $c_{\ell}\leq k \leq r-\ell$ then $(k, \ell)\in B_{\ell}$. Finally if $k>r-\ell$ then $r-k< T$ and $(k, \ell)\in B_{r-k}$.
\\
We now assume that $\ell \geq T$. Then $T\le u-r$ and therefore $T=\lfloor\frac{r-1}{2}\rfloor+1=\lfloor \frac{r+1}{2}\rfloor$. Recall from the definition and Lemma \ref{tableinvariantslem}(b) that $k_0=1$ and that $(k_0,k_1, \cdots, k_{t_{\max}})$ is a nondecreasing sequence. If $k_{t-1}\le k<k_t$ for some $t\in \{1,2,\ldots t_\max\}$ then 
$(k,\ell)\in A_t$. If $k_{t_\max}\le k\le r-T$ then $k_{t_\max}=\frac{r}{2}=T$ by Lemma \ref{tableinvariantslem}(a) and $(k, \ell) \in A_{T}$. If $k> r-T$ then $r-k\leq T-1$ and  $(k, \ell)\in B_{r-k}$.
\\
 We have shown that the elements of the following union of three sets
 \begin{equation}\label{covereq}
 \{A_t:1\le t\le T\}\bigcup \{B_t: 1\le t< T\}\bigcup \{ C_t: 1\le t<T\}
\end{equation}
are subsets of and cover the rectangle $\{1,2,\ldots ,r-1\}\times \{1,2,\ldots ,u-r\} $ in $\mathbb N\times\mathbb N$. By inspection, one checks that any two of the sets in the set \eqref{covereq} are disjoint. This completes the proof of assertion (1).
\vskip 0.2cm\noindent
We proceed with the proof of assertion (2). 
Let $(k,\ell)$ and $(k',\ell')$ be two distinct elements of $\{1,\ldots ,r-1\}\times \{1,\ldots ,u-r\}$.
Put 
\begin{align*}\mathcal F_A&=\cup\{A_t\mid 1\le t \leq T\}\\
\mathcal F_B&=\cup\{B_t\mid 1\le t<T\}\\
\mathcal F_C&=\cup\{C_t\mid 1\leq t<T\}.
\end{align*}
It follows from parts (a),(b),(c) of Theorem \ref{tablethm} that if $(k,\ell) $ and $(k',\ell')$ do not belong to the same set
 $\mathcal F_A,\mathcal F_B,\mathcal F_C$, then $P_{k,\ell}\not=P_{k',\ell'}$ since the two partitions are of different types.\par

\noindent{\bf Case 1}.  Suppose $(k,\ell)$ and $(k',\ell')\in \mathcal F_A$. Then $P_{k,\ell}=\left([u]^{k},[u-r]^{\ell}\right)$ and $P_{k',\ell'}=\left([u]^{k'},[u-r]^{\ell'}\right)$ are obviously distinct. 
\vskip 0.2cm\noindent
{\bf Case 2}. Suppose now that $(k,\ell)$ and $(k',\ell')$ are both in $\mathcal F_C$.  Then $(k, \ell)\in C_{\ell}$ and $(k', \ell')\in C_{\ell'}$. We also have 
$$\begin{array}{ll}P_{k, \ell}&=([u-r+2\ell]^{\ell}, [u-2\ell-d_{\ell}(q_{\ell}-1)]^{k-d_{\ell}}, (q_{\ell}-1)^{d_{\ell}}) \mbox{ and} \\ \\ P_{k, \ell}&=([u-r+2\ell']^{\ell'}, [u-2\ell'-d_{\ell'}(q_{\ell'}-1)]^{k-d_{\ell'}}, (q_{\ell'}-1)^{d_{\ell'}}). \end{array}$$
\begin{quote}
{\bf Case 2.1}. Suppose $\ell = \ell'$.  Then since $(k,\ell)\not=(k',\ell')$ we must have $k\neq k'.$ Thus for $t=\ell=\ell'$ we have
\begin{equation*}
[u-2t-d_t(q_t-1)]^{k-d_t}\not=[u-2t-d_t(q_t-1)]^{k'-d_t},
\end{equation*}
and so $P_{k,\ell}\not=P_{k',\ell'}$.
\end{quote}
\vskip 0.2cm
\begin{quote}
\noindent{\bf Case 2.2.} Suppose $\ell'<\ell$. Recall that \eqref{qtdteq} implies that
$$\begin{array}{ll}
[u-r+2\ell]^{\ell}&=\Big( (q_{\ell}+2)^{\ell-d_{\ell}}, (q_{\ell}+1)^{d_{\ell}} \Big)\\ \\

[u-r+2\ell']^{\ell'}&=\Big( (q_{\ell'}+2)^{\ell'-d_{\ell'}}, (q_{\ell'}+1)^{d_{\ell'}} \Big).
\end{array}$$ 

So to show that $P_{k,\ell}\not=P_{k',\ell'}$ it is enough to show that the rectangular partitions $(q_{\ell}+2)^{\ell-d_{\ell}}$ and $(q_{\ell'}+2)^{\ell'-d_{\ell'}}$ are distinct.
This is obvious if $q_{\ell'}>q_{\ell}$. Assume that $q_{\ell'}=q_{\ell}=q$. By Definition \ref{tablenotationdef}, we have 
$$u-r=\ell'q-d_{\ell'}=\ell q-d_{\ell}.$$ 

Then $(d_{\ell}-d_{\ell'})=q(\ell-\ell'),$ and therefore $(\ell'-d_{\ell'})-(\ell-d_{\ell})=(q-1)(\ell-\ell').$

Then since $\ell'<\ell\leq u-r$, we must have $q>1$ by the definition of $q=q_{\ell'}$ and consequently $(\ell'-d_{\ell'})-(\ell-d_{\ell})>0$. This in particular implies that $(q+2)^{\ell-d_\ell}$ and $(q+2)^{\ell'-d_{\ell'}}$ and therefore $P_{k,\ell}$ and $P_{k',\ell'}$ are distinct partitions.
\end{quote}
\vskip 0.2cm\noindent
{\bf Case 3}. 
Suppose now that $(k,\ell)$ and $(k',\ell')$ are both in $\mathcal F_B$.  
\vskip 0.2cm
{\bf Case 3.1} Suppose that $\ell=\ell'$ and $k'<k$. \par\par
\begin{quote}
{\bf Case 3.1.1.} Suppose that $k'<k\leq r-\ell$, then $(k,\ell)$ and $(k',\ell')$ are both in $B_{\ell}$. Then $P_{k,\ell}=\Big([u-r+2\ell]^\ell, [u-2\ell]^{k}\Big)$ and $P_{k',\ell'}=\Big([u-r+2\ell]^\ell, [u-2\ell]^{k'}\Big)$. Since $k\neq k'$, we obviously get $P_{k,\ell}\neq P_{k',\ell'}$ in this case. 
\end{quote}

\begin{quote}
{\bf Case 3.1.2.} Suppose that $k'\leq r-\ell<k$, then $(k',\ell')\in B_{\ell}$ and $(k, \ell)\in B_{r-k}$. Thus $$P_{k',\ell'}=\Big([u-r+2\ell]^\ell, [u-2\ell]^{k'}\Big)$$ and $$P_{k,\ell}=\Big([u-r+2(r-k)]^{r-k}, [u-2(r-k)]^{k+\ell-(r-k)}\Big).$$ 

By definition of $q_{\ell}$ we have $[u-r+2\ell]^\ell=\Big((q_{\ell}+2)^{\ell-d_{\ell}}, (q_{\ell}+1)^{d_{\ell}} \Big)$. On the other hand, the biggest part of $[u-r+2(r-k)]^{r-k}$ is $\lceil\frac{u-r}{r-k}\rceil+2$. So $P_{k,\ell}$ and $P_{k', \ell'}$ are obviously distinct when $\lceil\frac{u-r}{r-k}\rceil\neq q_{\ell}$. We assume next that $\lceil\frac{u-r}{r-k}\rceil = q_{\ell}=q$. Then there exists an 
integer $e$ such that $0\leq e <r-k$ and $u-r=(r-k)q-e$. Then by (2.1) $[u-r+2(r-k)]^{r-k}=\Big((q+2)^{r-k-e}, (q+1)^e\Big).$  On the other hand since $u-r=\ell q-d_\ell$ we also have $$\ell-d_\ell=(r-k-e)-(q-1)(\ell-(r-k)).$$  Since $r-k<\ell\leq u-r$, we also have $q>1$ and $\ell-(r-k)>0$. Thus $\ell-d_{\ell}<r-k-e$. Consequently, $P_{k,\ell}$ and $P_{k',\ell'}$ have different numbers of parts of size $q+2$, so they are distinct.
\end{quote}

\begin{quote}
{\bf Case 3.1.3} Suppose that $r-\ell<k'<k$, then $(k',\ell')\in B_{r-k'}$ and $(k, \ell)\in B_{r-k}$. Thus $$P_{k',\ell'}=\Big([u-r+2(r-k')]^{r-k'}, [u-2(r-k')]^{k'+\ell-(r-k')}\Big)$$ and $$P_{k,\ell}=\Big([u-r+2(r-k)]^{r-k}, [u-2(r-k)]^{k+\ell-(r-k)}\Big).$$ 

Then an argument very similar to the one for case 3.1.2 shows that either the biggest parts of $P_{k,\ell}$ and $P_{k',\ell'}$, or the number of the second biggest parts of the two partitions are different and therefore $P_{k,\ell}$ and $P_{k',\ell'}$ are distinct.
\end{quote}

\vskip 0.2cm
{\bf Case 3.2}.  Suppose that $\ell'<\ell$.  If $k+\ell\neq k'+\ell'$ then $P_{k,\ell}$ and $P_{k',\ell'}$ have different number of parts and therefore are distinct. For the rest of the proof of this case we assume that $k+\ell=k'+\ell'$. This in particular implies that $k'>k$. 

We note that if $\ell'\leq r-k'$ and $\ell\leq r-k$ then $(k,\ell)\in B_\ell$ and $(k',\ell')\in B_{\ell'}$. If $\ell'\leq r-k'<r-k<\ell$ then $(k',\ell') \in B_{\ell'}$ and $(k,\ell)\in B_{r-k}$. If $r-k'<\ell'<\ell \leq r-k$ then $(k',\ell') \in B_{r-k'}$ and $(k,\ell)\in B_\ell$. And if $r-k'<\ell'$ and  $r-k<\ell$ then $(k',\ell') \in B_{r-k'}$ and $(k,\ell)\in B_{r-k}$. Using the argument in the proof of case 2.2, we can show that if $\tau'<\tau$ then $[u-r+2\tau]^{\tau}\neq [u-r+2\tau']^{\tau'}$. So in all cases listed above, the top almost rectangular partitions of $P_{k,\ell}$ and $P_{k',\ell'}$ are different. Therefore  $P_{k,\ell} \neq P_{k',\ell'}$.This completes the proof of part (d) of Theorem \ref{tablethm}.\end{proof}
\begin{proof}[Proof of Corollary \ref{hookcor}] \label{proof3.15}
The Corollary follows immediately from the proof of Assertion (1) in the proof of Theorem \ref{tablethm}(d) and Definition \ref{tabledef}.
\end{proof}
\begin{example}[Table $\mathcal T(Q)$ and table invariants for $Q=(27,3)$]\label{27,3ex}
Here $u-r=3$, and $r=24$. We have $t_\max=\min\{u-r,\lfloor \frac{r-1}{2}\rfloor\}=3$ and the table invariants of Definition \ref{tablenotationdef} are\par
$$\begin{array}{|c|c|c|}
\hline
&&
\\
\begin{array}{ll}
q_1&=\lceil\frac{3}{1}\rceil=3\\ \\
d_1&=0, \\ \\
k_1&=\lceil \frac{27}{4} \rceil=7, \mbox{ and}\\ \\
c_1&=7.

\end{array}
&
\begin{array}{ll}
q_2&=\lceil\frac{3}{2}\rceil=2\\ \\
d_2&=1, \mbox{ and}\\ \\
k_2&=\lceil \frac{27}{3} \rceil=9, \mbox{ and}\\ \\
c_2&=\lceil \frac{24}{2}\rceil=12.
\end{array}
&
\begin{array}{ll}
q_3&=\lceil\frac{3}{3}\rceil=1\\ \\
d_3&=0, \\ \\
k_3&=\lceil \frac{25}{2} \rceil=13, \mbox{ and}\\ \\
c_3&=13.

\end{array}
\\
&& \\
\hline
\end{array}$$

\begin{table}[ht]
\large
$$\begin{array}{|l|l|l|}
\hline 

(27,3)&(27,[3]^2)&(27,[3]^3)\\
\hline
([27]^2,3)&([27]^2,[3]^2)&([27]^2,[3]^3)\\
\hline 
\vdots \hspace{.4 in} \vdots\hspace{.4 in} \vdots&\vdots \hspace{.4 in} \vdots\hspace{.4 in} \vdots&\vdots \hspace{.3 in} \vdots\hspace{.3 in} \vdots\\
\hline
([27]^6,3)&([27]^6,[3]^2)&([27]^6,[3]^3)\\
\hline
{\cellcolor{light-gray}(5, [25]^{7})}&([27]^7, [3]^{2})&([27]^7, [3]^{3})\\
\hline
{\cellcolor{light-gray}(5, [25]^{8})}&([27]^8, [3]^{2})&([27]^8, [3]^{3})\\
\hline
{\cellcolor{light-gray}(5, [25]^{9})}&{\cellcolor{light-pink}\color{blue}{([7]^2, [22]^{8}, 1)}}&([27]^9, [3]^{3})\\
\hline
{\cellcolor{light-gray}(5, [25]^{10})}&{\cellcolor{light-pink}\color{blue}{([7]^2, [22]^{9}, 1)}}&([27]^{10}, [3]^{3})\\
\hline
{\cellcolor{light-gray}(5, [25]^{11})}&{\cellcolor{light-pink}\color{vivid-blue}{([7]^2, [22]^{10}, 1)}}&([27]^{11}, [3]^{3})\\
\hline
{\cellcolor{light-gray}(5, [25]^{12})}&{\cellcolor{med-pink}{\color{vivid-blue}([7]^2, [23]^{12})}}&([27]^{12}, [3]^{3})\\
\hline
{\cellcolor{light-gray}(5, [25]^{13})}&{\cellcolor{med-pink}{\color{vivid-blue}([7]^2, [23]^{13})}}&{\cellcolor{aqua}{\color{purple}([9]^{3}, [21]^{13})}}\\
\hline
{\cellcolor{light-gray}(5, [25]^{14})}&{\cellcolor{med-pink}{\color{vivid-blue}([7]^2, [23]^{14})}}&{\cellcolor{aqua}{\color{purple}([9]^{3}, [21]^{14})}}\\
\hline
{\cellcolor{light-gray}\vdots \hspace{.4 in} \vdots\hspace{.4 in} \vdots}&{\cellcolor{med-pink}\vdots \hspace{.4 in} \vdots\hspace{.4 in} \vdots}&{\cellcolor{aqua}\vdots \hspace{.3in} \vdots\hspace{.3 in} \vdots}\\
\hline
{\cellcolor{light-gray}(5, [25]^{20})}&{\cellcolor{med-pink}{\color{vivid-blue}([7]^2, [23]^{20})}}&{\cellcolor{aqua}{\color{purple}([9]^{3}, [21]^{20})}}\\
\hline
{\cellcolor{light-gray}(5, [25]^{21})}&{\cellcolor{med-pink}{\color{vivid-blue}([7]^2, [23]^{21})}}&{\cellcolor{aqua}{\color{purple}([9]^{3}, [21]^{21})}}\\
\hline
{\cellcolor{light-gray}(5, [25]^{22})}&{\cellcolor{med-pink}{\color{vivid-blue}([7]^2, [23]^{22})}}&{\cellcolor{med-pink}{\color{vivid-blue}([7]^2, [23]^{23})}}\\
\hline
{\cellcolor{light-gray}(5, [25]^{23})}&{\cellcolor{light-gray}(5, [25]^{24})}&{\cellcolor{light-gray}(5, [25]^{25})}\\
\hline
\end{array}$$
\caption{Table $\mathcal T(Q), Q=(27,3)$}\label{27,3table}
\end{table}

\normalsize
Recall from Theorem \ref{tablethm} that $k_1=7,k_2=9,k_3=13$ are the rows of  $\mathcal T(Q)$ where B/C hooks begin. By Theorem~\ref{tablethm} we have
 
\begin{itemize} 
\item For $1\leq k\leq 6$, and $1\leq \ell \leq 3$, we have $P_{k,\ell}=([27]^k, [3]^\ell).$  (Type A)
\item For $\ell=1$ and $7\leq k \leq 23$, we have $P_{k,1}=(5, [25]^{k})$. (Type B, vertical part of B/C hook, in green)
\item For $k=23$, and $2 \leq \ell\leq 3$, we have $P_{23,\ell}=(5, [25]^{\ell+22})$. (Type B, horizontal part of the same B/C hook, also in green)
\item For $7\leq k\leq 8$, and $2\leq \ell \leq 3$, we have $P_{k,\ell}=([27]^k, [3]^{\ell}).$ (Type A)
\item For $\ell=2$ and $9 \leq k \leq 11$, we have $P_{k,2}=([7]^2, [22]^{k-1}, 1)$. (Type C, in light pink in Table~\ref{27,3table})
\item For  $\ell=2$ and $12\leq k\leq 22$, we have $P_{k, \ell}=([7]^2, [23]^{k+\ell-2})$. (Type B, dark pink in Table~\ref{27,3table}, vertical part of the B/C hook that started with light pink type C partitions)

\item $P_{22,3}=([7]^2, [23]^{23})$. (Type B in Table~\ref{27,3table}, the horizontal part of the B/C hook that started with light pink type C partitions).
\item For $9\leq k \leq12$ and $\ell=3$ we have $P_{k,3}=([27]^{k}, [3]^3).$ (Type A)
\item For $\ell=3$ and $13\leq k \leq 21$, $P_{k,3}=([9]^3, [21]^{k}).$ (Type B, this vertical B/C hook is in aqua in Table \ref{27,3table}). 

\end{itemize}
\end{example}
\subsection{Alternating pattern for $\mathcal T(Q)$.}
We say that the table $\mathcal T(Q)$ (or, for short, $Q$) has \emph{alternating pattern} if both of the following conditions are satisfied:
\begin{enumerate}[i.]
\item All partitions in $\mathcal T(Q)$ are of type A or B, according to Definition \ref{casesdef}.
\item $k_t=t+1$ for every $t\in\{1,\ldots t_\max\}$, (or $r=2$).
\end{enumerate}
We say $\mathcal T(Q)$ (or, for short, $Q$) is \emph{weakly alternating} if (ii) is satisfied: this is equivalent to the $t$-th B/C hook begins  with $P_{t+1,t}$ for each $t\in\{1,\ldots, t_\max\}$.  This is also equivalent to $A_t=\{(k,\ell)\mid k=k_{t-1}, t\le \ell\le u-r\}$ for each such $t$.
R. Zhao showed in \cite{Z} that  $\mathcal T(Q)$ has alternating pattern when $u\gg r$. We give a proof using Theorem \ref{tablethm}.
\begin{corollary}\label{normalpatterncor}
Let $u>r\ge 2$. If $ Q=(u,u-r)$ is weakly alternating, then $u\ge\max\{2r-2, \frac{3r-1}{2}\}$ and $ Q'=(u+1,u+1-r)$ is weakly alternating as well.  If $u\ge {r+r^2/8}$ then $Q=(u,u-r)$ has alternating pattern.  
\end{corollary}
\begin{proof}
When $r=2$ then $\mathcal T(Q)$ consists of a single row of type A partitions, there are no B/C hooks, so condition (ii) is (vacuously) satisfied and $Q$ has alternating pattern. Also, $Q'=(u+1,u+1-r)$ has again $r'=(u+1)-(u+1-r)=2$, so $Q'$ also has alternating pattern.  We assume henceforth that $r\ge 3$. We note\footnote{We thank a referee for pointing this out, thus removing an unnecessary assumption $u\ge 3r/2$ from a previous statement of the Corollary.} that 
\begin{equation}\label{u3r2eq}
u\ge 3r/2,
\end{equation}
since when $r\ge 4 $ then $r+r^2/8\ge 3r/2$, and when $r=3$ and $u$ is an integer, each of $u\ge r+r^2/8$ and $u\ge 3r/2$ is equivalent to $u\ge 5$.\par
By definition $q_1=(u-r)$ so by \eqref{kalternative} $k_1=1+\lceil \frac{r-1}{u-r+1}\rceil$. Therefore $k_1=2$ is equivalent to $u-r+1\ge r-1$, so to $u\ge 2r-2$.\par We claim that  $k_{t_\max}=t_{max}+1$ is equivalent to $u\ge \frac{3r-1}{2}$. When $u>\frac{3r-1}{2}$ then $u-r>\lfloor \frac{r-1}{2}\rfloor$ and $k_{t_\max}=t_{max}+1$ is \eqref{tMaxcase2eq} in the proof of Lemma \ref{tableinvariantslem}. When $ u= \frac{3r-1}{2}$ the integer $r$ is odd, and we have by Lemma \ref{tableinvariantslem}(a) $k_{t_\max}= \lceil \frac{r+1}{2}\rceil= \lceil \frac{r-1}{2}\rceil+1=t_\max+1$. But when $u<\frac{3r-1}{2}$ we have
$u-r<\frac{r-1}{2}$ and conclude that
\begin{equation}\label{gapeq}
k_{t_\max}=\big\lceil\frac{r+1}{2}\big\rceil\ge u-r+2=t_\max+2.
\end{equation}This completes the proof of the claim, and the assertion $u\ge \max\{2r-2,\frac{3r-1}{2}\}$ when $Q$ is weakly alternating (the case $r=2$ being trivial).\par
 Since $q_t=\lceil\frac{u-r}{t}\rceil$ is nondecreasing in $u$ the equation \eqref{kalternative} shows that for a fixed pair $(t,r)$, the integer $k_t$ is nonincreasing as $u$ increases and becomes stable once it reaches $k_t=t+1$.  Note also that if $Q=(u,u-r)$ is weakly alternating then $u\ge\frac{3r-1}{2}$ and this implies ${t_\max}=\lfloor\frac{r-1}{2}\rfloor$ for both $Q$ and $Q'=(u+1,u+1-r)$. Thus (ii) for $Q$ implies (ii) for $Q'$. This finishes the proof of the statements about the weak-alternating property.\par
To show the statement about the alternating property, note that
Condition (i.) follows from Corollary \ref{CCor}. We now show Condition (ii.). As noted we may assume $r\ge 3$ and by \eqref{u3r2eq} we have $u-r\ge r/2$. It follows that $t_\max=\big\lfloor \frac{r-1}{2}\big\rfloor$. We have $q_t=\lceil\frac{u-r}{t}\rceil$ and $k_t=t+\lceil \frac{r-2t+1}{q_t+1}\rceil$. So we need to show that $q_t\ge r-2t$ for every $t\in \{ 1,2, \ldots, t_\max\}$. This holds when  $\frac{u-r}{t} \ge r-2t$, or
\begin{equation}\label{normalineq}
u\ge r+t(r-2t).
\end{equation}
The expression on the right side of \eqref{normalineq} has a maximum of $r+r^2/8$ at $t=r/4$. This proves that Condition (ii.) holds when $u\ge r+r^2/8$ and completes the proof of the Corollary.
\end{proof}\par\noindent
\begin{comment}
{\bf Comparison with B/C hooks of a general table.} 
For all tables the lower left corner of the $t$-th B/C hook is $P_{r-t,t}, 1\le t\le t_\max$.
The leading element of the $t$-th B/C hook is $P_{k_t,t}$ on the $t+k_t$ lower right to upper left diagonal of $\mathcal T(Q)$: so the leading elements of B/C hooks lie on a choice of $t_\max$ such diagonals beginning with $1+k_1=1+\lceil \frac{u}{u-r+1}\rceil$ and ending with $t_\max +k_{t_\max}\le r.$  In  a weakly alternating pattern these leading B/C elements are $P_{k,\ell}$ for  $(k,\ell)\in\{(2,1),(3,2),\dots,( t_\max+1,t_\max).\}$.
\end{comment}
\subsection{Completeness of the table $\mathcal T(Q)$.}\label{tablecompletenesssec}
In this section we will prove Theorem~\ref{table2thm}, which asserts that $\mathcal T(Q)$ is all of $\mathfrak Q^{-1}(Q)$. This will complete the proof of the Table Theorem \ref{Zthm}. R. Zhao proved this for $u\gg r$ in \cite{Z}.
\begin{lemma}\label{Ccompletelem} Fix $Q=(u,u-r), u>r\ge 2$. All type C partitions $P$ that satisfy $\mathfrak Q(P)=Q$ and that are not of type A or B, occur in the table $\mathcal T(Q)$ of Definition \ref{tabledef}.
\end{lemma}
\begin{proof}
From Lemma \ref{Clem} we know that 
\begin{align}
\{P_C&:  C=({\sf c}_1,{\sf c}_2,s_1,s_2;a) \text { with } {\sf c}_1,{\sf c}_2,s_1,s_2\in \mathbb Z_{>0} \text { and } a\ge 4\} \text { where}\notag\\
P_C&=(a^{{\sf c}_1}, (a-1)^{{\sf c}_2+s_2}, (a-2)^{{\sf c}_1+s_1}, (a-3)^{{\sf c}_2})\label{eqsetPC}
\end{align}
is the set of all partitions that are of type C but not of type A or B.\par
Let $C=({\sf c}_1,{\sf c}_2,s_1,s_2;a)$ with $ {\sf c}_1,{\sf c}_2,s_1,s_2\in \mathbb Z_{>0}$ and $ a\ge 4$. Assume that $\mathfrak Q(P_C)=(u,u-r)$. 
We need to prove that the partition $P_C$ is in the table $\mathcal T(Q)$.  By Lemma \ref{Clem}  $\mathfrak Q(P_C)=(u,u-r)$ implies that
the following equalities hold:
\begin{align}
  u&=(a-2)({\sf c}_1+{\sf c}_2+s_1+s_2)+{\sf c}_2+s_2+2{\sf c}_1\label{eq:ua2}\\
 &=(a-1)({\sf c}_1+{\sf c}_2+s_1+s_2)+{\sf c}_1-s_1;\label{eq:ua1}\\
  u-r&=(a-2)({\sf c}_1+{\sf c}_2)-{\sf c}_2\label{eq:ur}.
\end{align}
We need to show that there exists $t\in\{1,\ldots t_\max\}$ and $(k,\ell)\in C_t$ where $C_t$ is as in Theorem~\ref{tablethm}(b), such that $P_{k,\ell}=P_C$.
Recall from  \eqref{tableCeq} in the proof of Theorem~\ref{tablethm} that if $(k,\ell)\in C_t$ then
\begin{equation}\label{tableCbeq}
P_{k,\ell}=((q_t+2)^{t-d_t},(q_t+1)^{d_t+n_1},q_t^{n_0}, (q_t-1)^{d_t})
\end{equation}
for certain positive integers $n_0$ and $n_1$. It follows from \eqref{eqsetPC} that for $P_{k,\ell}=P_C$ to hold we must have
\begin{equation}\label{eq:d}
t-d_t={\sf c}_1 \text { and } d_t={\sf c}_2.
\end{equation}
In particular, the unique candidate for $t$ is
\begin{equation}\label{eq:t}
t={\sf c}_1+{\sf c}_2.
\end{equation}
It follows from \eqref{PCeqn} and \eqref{eq:ur} that ${\sf c}_1+{\sf c}_2\le \min\{u-r,\lfloor\frac{r-1}{2}\rfloor\}$, whence $1\le t\le t_\max$. 
 From \eqref{eq:ua1},  \eqref{eq:ur},\eqref{tableCbeq} and \eqref{eq:t} we obtain that, by Definition \ref{tablenotationdef},
 \begin{align}
 q_t&=\Big\lceil \frac{u-r}{t}\Big\rceil=a-2; \text { and}\notag\\
 k_t&=\Big\lceil \frac{u+1-t+d_t}{q_t+1}\Big\rceil=\Big\lceil\frac{\left({\sf c}_1+{\sf c}_2+s_1+s_2\right)(a-1)+1-s_1}{a-1}\Big\rceil.\label{3.30eq}
 \end{align}
 Since $s_1 \geq 1$, we have $k_t\leq {\sf c}_1+{\sf c}_2+s_1+s_2$.
 
 On the other hand, since $d_t=\sl {\sf c}_2\ge 1$ equation \eqref{eq:ua2} and Definition \ref{tablenotationdef} yield that
 \begin{align}
 c_t&=\Big\lceil\frac{u-2t+d_t}{q_t}\Big\rceil\\
&={\sf c}_1+{\sf c}_2+s_1+s_2+\Big\lceil\frac{s_2}{a-2}\Big\rceil .\label{3.31eq}
 \end{align}
Since $s_2\ge 1$ we also have  $c_t> \left({\sf c}_1+{\sf c}_2+s_1+s_2\right)$.
This in particular implies that $k_t<c_t$. So the set $C_t$ is not empty. 

Let $k=\left({\sf c}_1+{\sf c}_2+s_1+s_2\right)$. Then $k_t\leq k<c_t$. Thus by definition of $C_t$, we have $(k, t)\in C_t.$ To complete the proof we will show that $P_{k,t}=P_C$.

Since $q_t=a-2$ and $k={\sf c}_1+{\sf c}_2+s_1+s_2$, using \ref{eq:ua2}, \ref{eq:d} and \ref{eq:ur}, we get
$$\begin{array}{rl}
u-2t-d_t(q_t-1)&=(a-2)k+s_2-{\sf c}_2-{\sf c}_2(a-3)\\
&=(a-2)(k-{\sf c}_2)+s_2\\
&=(a-2)(k-d_t)+s_2\\ \\ 
\mbox{and}\\ \\
u-r+2t&=(a-2)({\sf c}_1+ {\sf c}_2)-{\sf c}_2+2({\sf c}_1+{\sf c}_2)\\
&=a({\sf c}_1+ {\sf c}_2)-{\sf c}_2\\
&=at-{\sf c}_2.
\end{array}$$
This implies that 
$$\begin{array}{rl}
[u-2t-d_t(q_t-1)]^{k-d_t}&=\Big((a-1)^{s_2},(a-2)^{k-d_t-s_2}\Big)\\&=\Big((a-1)^{s_2},(a-2)^{s_1+{\sf c}_1}\Big)\\ 

\mbox{and}\\

[u-r+2t]^{t}&=\Big( a^{{\sf c}_1}, (a-1)^{{\sf c}_2}\Big).
\end{array}$$

Therefore $$\begin{array}{ll}P_{k, t}&=\Big([u-r+2t]^{t},[u-2t-d_t(q_t-1)]^{k-d_t}, (q_t-1)^{d_t}\Big)\\&=(a^{{\sf c}_1}, (a-1)^{{\sf c}_2+s_2}, (a-2)^{{\sf c}_1+s_1}, (a-3)^{{\sf c}_2})=P_C\end{array}.$$

 \end{proof}

We can now prove the completeness part of the Table Theorem \ref{Zthm}.
\begin{theorem}[Part II of Table Theorem]\label{table2thm} Let $Q=(u,u-r)$, $u> r\ge 2$. The table $\mathcal T(Q)$ of Definition \ref{tabledef} contains all the partitions in $\mathfrak{Q}^{-1}(Q)$.
\end{theorem}
\begin{proof}  Lemma \ref{Ccompletelem} shows the completeness for type C partitions. We next consider type A and then type B partitions.\vskip 0.2cm
{\bf Claim A}: Let $P = (a^{n_a}, (a-1)^{n_{a-1}}, b^{n_b}, (b-1)^{n_{b-1}})$ be a partition as in \eqref{Pabeq}. In particular $a-2\ge b\ge 1, n_a>0,n_b>0$ and $n_{b-1}=0$ if $b=1$. Suppose that ${\mathfrak Q}(P) = (u,u-r)$ and that $P$ is of type A.
 Put
\begin{equation}\label{definetkleq}
\begin{array}{ll}
k&= n_a+n_{a-1}; \\ 
\ell &= n_b+n_{b-1};\\
T&=\min\{u-r, \lfloor\frac{r-1}{2}\rfloor\}+1=t_\max+1;\\
t &= \left\{\begin{array}{ll} \text { the unique $t\in\{ 1,\ldots, t_\max\},$ satisfying}& k_{t-1}\le k<k_t \text { if }k<k_{t_\max}\\ \\
T&\qquad\qquad\qquad\mbox{                    if }  k\ge k_{t_\max}.\end{array}\right.
\end{array}
\end{equation}

Then $(k,\ell) \in A_t$ and $P_{k,\ell} = P$.
\bigskip

\noindent {\it Proof of Claim A.}
By Definition \ref{casesdef} we have 
\begin{align}
u&=an_a+(a-1)n_{a-1},\label{3.45bis} \\
u-r&=bn_b+(b-1)n_{b-1}.\label{3.46bis}
\end{align}\noindent
This in particular gives $P=([u]^k,[u-r]^\ell)$, where
\begin{equation}\label{Pstdeq}
 [u]^k=(a^{n_a}, (a-1)^{n_{a-1}}) \text{ and }[u-r]^{\ell}=(b^{n_b}, (b-1)^{n_{b-1}}).
 \end{equation}\noindent
Since $P$ is of type A we must also have $|C_b|\le u$, so
 $$bn_b+(b-1)n_{b-1}+2(n_a+n_{a-1})\le u.$$ 
 Therefore by \eqref{3.46bis} we have 
\begin{equation}\label{kr2eq}
k=n_a+n_{a-1}\le \frac{r}{2}.
\end{equation} 
\vskip 0.2cm\noindent
$\,$ {\bf Case i}.  Assume $k=n_a+n_{a-1}\ge k_{t_\max}$.  We use Lemma \ref{tableinvariantslem}(a). If $k_{t_\max}=\lceil \frac{r+1}{2}\rceil$ then we have from \eqref{kr2eq}
 $$\big\lceil\frac{r+1}{2}\big\rceil=k_{t_\max}\le k=n_a+n_{a-1}\le  \frac{r}{2},$$
 a contradiction. So $r$ is even (we write $r=2r'$), $k_{t_\max}=r'$ and $u-r> \lfloor\frac{r-1}{2}\rfloor=r'-1 $. Then by the definition of $t_\max$ we have $t_\max=r'-1$ and $T=r'$.
 From \eqref{kr2eq} we obtain that
 \begin{equation}
 k_{t_\max}=r'\le k\le r'=r-T.
 \end{equation}
Evidently, from \eqref{Pstdeq} $\ell=n_b+n_{b-1}\le u-r$.  We need to show $T\le \ell$.  Assume by way of contradiction that $T>\ell$, that is, $\ell\le r'-1$. Since $k=r', 2r'=r$ and $\lceil \frac{u}{k}\rceil=a$ we know that $\lceil \frac{u-r}{k}\rceil=a-2$ and therefore that $\lceil \frac{u-r}{\ell}\rceil\ge a-2$. If  $\lceil \frac{u-r}{\ell}\rceil > a-2$ then we have a contradiction with \eqref{Pstdeq}, so we may assume that 
$\lceil \frac{u-r}{\ell}\rceil =a-2$. Since $[u]^k=[u]^{r'}=\left(a^{n_a},(a-1)^{n_{a-1}}\right)$ we know that $[u-r]^{r'}=[u-2r']^{r'}=\left((a-2)^{n_a},(a-3)^{n_{a-1}}\right)$. Consequently, since $\ell\le r'-1$, $[u-r]^{\ell}=((a-2)^{n_a+s},(a-3)^{\ell-n_a-s})$ with $s>0$, and, by \eqref{Pstdeq} $b=a-2$ and $n_b=n_a+s$. It follows that $|U_{middle}|=|C_{a-1}|$ satisfies
\begin{align*}
|C_{a-1}|&=n_{a-1}(a-1)+n_b(a-2)+2n_a\\
&=n_a(a)+n_{a-1}(a-1)+s(a-2)=u+s(a-2)>u \text { by \eqref{3.45bis} },
\end{align*}
which contradicts that the largest part of $\mathfrak Q(P)$ is $u$.

We have shown $T\le \ell$.  
By the definition of $A_T$ in Theorem \ref{tablethm} this shows that when $n_a+n_{a-1}\ge k_{t_\max}$, then $(k, \ell)\in A_T$ and $P_{k,\ell} = P$, as claimed. This completes the proof of Claim A, case i. 
\vskip 0.2 cm
$\,$ {\bf Case ii}.
Now assume that $k=n_a+n_{a-1}<k_{t_\max}$.  Let $t$ be the unique integer such that $1\le t\le t_\max$ and $k_{t-1}\le k<k_t$. We have to show that $\ell\ge t$. To lighten notation in the rest of the proof,  we set $\tau=t-1$. So $0\le \tau\le t_\max -1 $ and $k_\tau\le k<k_{\tau+1}$; we need to show $\ell\ge \tau+1$.  If $\tau=0$ then $\ell\ge \tau+1 $ is true (vacuously). So we assume $\tau\ge 1$. To obtain a contradiction we assume that  $\ell \le \tau $. \par
We have by Definition \ref{tablenotationdef} and Equation \eqref{kalternative}
\begin{equation*}
k_\tau =\Big \lceil \frac{u+1-\tau+d_\tau}{q_\tau+1} \Big\rceil=\tau+\Big \lceil \frac{r-2\tau+1}{q_\tau+1} \Big\rceil.
\end{equation*}
Thus  \begin{equation} k_\tau(q_\tau+2)=k_\tau(q_\tau+1)+k_\tau\geq u+1-\tau+d_\tau+\tau>u. \end{equation}

Therefore $\lfloor \frac{u}{k}\rfloor\le \frac{u}{k_\tau} <q_\tau+2$, and consequently $\lfloor\frac{u}{k}\rfloor\leq q_\tau+1.$
 In particular we have
\begin{equation}\label{3.51eq}
[u]^{k}=(a^{n_a},(a-1)^{n_{a-1}}), \text { where } a-1=\big\lfloor{\frac{u}{k}}\big\rfloor \le q_\tau+1. 
\end{equation}
By \eqref{qtdteq} we have $[u-r]^{\tau}=\Big(q_\tau^{\tau-d_\tau}, (q_\tau-1)^{d_\tau}\Big).$
\\
Since $\ell\le \tau$ it follows that for the largest part $b$ of $[u-r]^\ell$ we have $b\ge q_\tau$. If $b>q_\tau$ then we have a contradiction with \eqref{Pstdeq}, so we assume that $b=\lceil \frac{u-r}{\ell}\rceil=q_\tau$. Since $a-b\ge 2$ this implies with 
\eqref{3.51eq} that $a=b+2$.
Since $\ell\le \tau$,
$$\begin{array}{ll}
 \tau-d_\tau&\le n_b \\
k\ge k_\tau&=\Big \lceil \frac{u+1-\tau+d_\tau}{b+1}\Big\rceil\ge\Big \lceil \frac{u+1-n_b}{b+1}\Big\rceil.
\end{array}$$
Since the largest part of $\mathfrak Q(P)$ is $u$ we know that the U-chain $C_{a-1}$ has length at most $u$, that is
$$(a-1)n_{a-1}+(a-2)n_b+2n_a\le u.$$
This by \eqref{3.45bis} implies that $n_b \le n_a$. So we have 

$$k\ge k_\tau\ge\Big \lceil \frac{u-n_b+1}{b+1}\Big\rceil\geq  \frac{u-n_b+1}{b+1}>\frac{u-n_a}{a-1}=n_a+n_{a-1}=k,$$
a contradiction. This completes the proof of Claim~A.
\par\vskip 0.2cm
{\bf Claim B}: Let $P = (a^{n_a}, (a-1)^{n_{a-1}}, b^{n_b}, (b-1)^{n_{b-1}})$ be a partition as in \eqref{Pabeq} such that $\mathfrak Q(P) = (u,u-r)$ and suppose that $P$ is of type B but not of type A. 
Recall, $n_a>0$ and $n_b>0$.
\\

(i) If $b=a-2$ and $n_{b-1}=0$ then put 
$$\begin{array}{ll}
t&=n_a,\\ 
k&=n_{a-1}+n_b \mbox{ and}\\
\ell&=t.
\end{array}$$
\medskip

(ii) If $b<a-2$ or $n_{b-1}\not= 0$ then put
$$\begin{array}{ll}
t&=n_a+n_{a-1},\\
k&=\min\{n_{b}+n_{b-1}, r-t\}, \mbox{ and}\\
\ell&=t+n_b+n_{b-1}-k.
\end{array}$$

Then $(k,\ell) \in B_t$ and $P_{k,\ell} = P$.

\bigskip

\noindent {\it Proof of Claim B.}\par\noindent
$\,${\bf Case i.}
If $b=a-2$ and $n_{b-1}=0$, then we can write $$P=(a^{n_a}, (a-1)^{n_{a-1}}, (a-2)^{n_{a-2}}).$$
where $P$ being of type B, but not of type A, is equivalent to  $n_{a-2}>n_a$. We have $a\ge 3$ (since $b=a-2> 0$)  and
\begin{align} u&=n_{a-1}(a-1)+n_{a-2}(a-2)+2n_a;\notag\\
u-r&=(a-2)n_a;\notag\\
r&=n_{a-1}(a-1)+n_{a-2}(a-2)-(a-4)n_a.\label{suiteeqn}
\end{align}
Evidently $t=n_a\le u-r$. We have $r\ge (a-2)(n_{a-2}-n_a)+2n_a$, and since $n_{a-2}>n_a$ it follows that $r>2n_a=2t$, whence $1\le t\le \min\{u-r,\lfloor\frac{r-1}{2}\rfloor\}$. Since $[u-r]^t=(a-2)^{t}$  Definition \ref{tablenotationdef} tells us that $q_t=a-2, d_t=0, c_t=k_t$ and since $q_t+1=a-1$ 
\begin{align}
k_t&=\Big\lceil \frac{u+1-t+d_t}{q_t+1}\Big\rceil=n_{a-1}+\Big\lceil\frac{1}{a-1}\left( (a-2)n_{a-2}+n_a+1\right)\Big\rceil\notag\\
&\le n_{a-1}+n_{a-2}=n_{a-1}+n_b.
\end{align}
Using that $a\ge3$, that $n_{a-2}>n_a$, and \eqref{suiteeqn} it is straightforward to show that $n_{a-1}+n_{a-2}\le r-t$. Consequently,
\begin{equation}
c_t  =k_t\le k\le r-t.
\end{equation}
Therefore $(k,\ell)$ is an element of  (the ``vertical portion'' of) $B_t$ from Theorem \ref{tablethm}, and one immediately checks that $P_{k,\ell}=P$.\vskip 0.2cm\noindent
$\,${\bf Case ii}.  Since by assumption $n_{b-1}=0$ if $b=1$, we have $a\geq 4$.  
\medskip
 We first show that the integer $t$ defined in case (ii) above satisfies $$1\leq t \leq  \min\big\{u-r, \big\lfloor\frac{r-1}{2}\big\rfloor\big\}.$$ 

Since $n_a+n_{a-1}>0$, it is obvious that $1\leq t$. 
On the other hand, by assumption $P$ is of type B but not of type A. Thus 
\begin{equation}\label{longbottom}
an_a+(a-1)n_{a-1}<u=bn_b+(b-1)n_{b-1}+2(n_a+n_{a-1}),
\end{equation}
and 
\begin{equation}\label{shorttop}
u-r=(a-2)n_a+(a-3)n_{a-1}.
\end{equation}

Thus by \eqref{longbottom} and \eqref{shorttop} we have
 \begin{align*}
 r&=u-(u-r)>2(n_a+n_{a-1})=2t, \text { so }\\
  t&\le \lfloor\frac{r-1}{2}\rfloor.
  \end{align*}
 Additionally, since $a-3\geq 1$, from (\ref{shorttop}) we also get  $$u-r\geq (a-3)(n_a+n_{a-1})\geq n_a+n_{a-1}.$$  
\noindent
So $1\leq t \leq  \min\{u-r, \lfloor\frac{r-1}{2}\rfloor\}$  as desired. 
\medskip

Next, we  show that $(k, \ell)\in B_t$ and $P_{k,\ell}=P$.  
Using \eqref{longbottom} and \eqref{shorttop}  we find that 
\begin{equation}\label{typebinvs}
\begin{array}{ll}
q_t&=\lceil\frac{u-r}{n_a+n_{a-1}}\rceil=a-2\\ \\
d_t&=n_{a-1}\\ \\
k_t&=\lceil\frac{u-n_a+1}{a-1}\rceil\\
\\
c_t&=\left\{
\begin{array}{lll}
\lceil\frac{bn_b+(b-1)n_{b-1}+n_{a}+1}{a-1}\rceil&&\mbox{ if } n_{a-1}=0\\ \\
\lceil\frac{bn_b+(b-1)n_{b-1}+n_{a-1}}{a-2}\rceil&&\mbox{ if } n_{a-1}>0.

\end{array}
\right.
\end{array}
\end{equation}
\noindent $\quad$
{\bf Case ii.1.} Suppose that $n_b+n_{b-1}\leq r-t$. Then by the definition of $k$ and $\ell$ in (ii), $k=n_b+n_{b-1}\leq r-t$ and $\ell=t$. In order to show that $(k,\ell)\in B_t$ in this case it is enough to show that $c_t\leq k$. 
\begin{quote}
{\bf Case ii.1.1.} Assume that $n_{a-1}=0$. Then $t=n_a$ and $c_t=k_t=\lceil\frac{bn_b+(b-1)n_{b-1}+n_{a}+1}{a-1}\rceil.$ By (\ref{longbottom}) we have $(a-2)n_a< bn_b+(b-1)n_{b-1}$. Therefore
$$\begin{array}{ll}
\frac{bn_b+(b-1)n_{b-1}+n_{a}+1}{a-1}&\leq \frac{bn_b+(b-1)n_{b-1}+\frac{bn_b+(b-1)n_{b-1}}{a-2}+1}{a-1}\\ \\
&=\frac{bn_b+(b-1)n_{b-1}}{a-2}+\frac{1}{a-1}\\ \\
&=\frac{b}{a-2}(n_b+n_{b-1})-\frac{n_{b-1}}{a-2}+\frac{1}{a-1}\\ \\
&=n_b+n_{b-1}-\frac{a-2-b}{a-2}(n_b+n_{b-1})-\frac{n_{b-1}}{a-2}+\frac{1}{a-1}\\ \\ 
&<n_b+n_{b-1}-\frac{a-2-b+n_{b-1}}{a-2}+\frac{1}{a-2}.\\\\
\end{array}$$

Using the assumption that $b<a-2$ or $n_{b-1}\neq 0$ we get $a-2-b+n_{b-1}\geq 1$. Thus, $\frac{bn_b+(b-1)n_{b-1}+n_{a}+1}{a-1}\leq n_b+n_{b-1}$.
 Since in this case $c_t=\lceil\frac{bn_b+(b-1)n_{b-1}+n_{a}+1}{a-1}\rceil$, this implies $c_t\leq k$, as desired. It is straightforward to check that $P_{k,\ell}=P$.

\end{quote}

\begin{quote}
{\bf Case ii.1.2.} Now suppose that $n_{a-1}\neq 0$. Then $t=n_a+n_{a-1}$ and \large $c_t=\lceil\frac{bn_b+(b-1)n_{b-1}+n_{a-1}}{a-2}\rceil.$\normalsize

By (\ref{longbottom}) we have 
$$(a-3)(n_a+n_{a-1}) \leq (a-3)(n_a+n_{a-1})+n_{a}<bn_b+(b-1)n_{b-1}\leq b(n_b+n_{b-1}).$$

 Therefore
$$
\begin{array}{ll}
\frac{bn_b+(b-1)n_{b-1}+n_{a-1}}{a-2}&=\frac{b(n_b+n_{b-1})-n_{b-1}+(n_{a}+n_{a-1})-n_a}{a-2}\\ \\
&<\frac{b(n_b+n_{b-1})+\frac{b}{a-3}(n_b+n_{b-1})-n_{b-1}-n_a}{a-2}\\ \\

& = \frac{b}{a-3}(n_b+n_{b-1})-\frac{n_a+n_{b-1}}{a-2}.
\end{array}
$$
\normalsize
If $b\leq a-3$, then it follows that $$c_t=\lceil \frac{bn_b+(b-1)n_{b-1}+n_{a-1}}{a-2}\rceil \leq n_b+n_{b-1}=k$$ as desired.

On the other hand, if $b=a-2$, then $P=\Big(a^{n_a}, (a-1)^{n_{a-1}}, (a-2)^{n_b}, (a-3)^{n_{b-1}}\Big)$. Since by assumption $P$ is of type B, and $n_{b-1}\not=0$ we in particular have $$(a-1)n_{a-1}+ (a-2){n_b}+2n_a\leq (a-2)n_b+(a-3)n_{b-1}+2(n_a+n_{a-1}).$$ Therefore $n_{a-1}\leq n_{b-1}$. We have 
\large

$$
\begin{array}{ll}
\frac{bn_b+(b-1)n_{b-1}+n_{a-1}}{a-2}&\leq \frac{b(n_b+n_{b-1})}{a-2}\\ \\
&= \frac{(a-2)(n_b+n_{b-1})}{a-2}=n_b+n_{b-1}.
\end{array}
$$
\normalsize

And we again obtain the desired inequality $c_t\leq k$. Using \eqref{longbottom}, \eqref{shorttop}, $\ell=t=n_a+n_{a-1}$ and $k=n_b+n_{b-1}$ it is straightforward to check that $P_{k,\ell}=P$.

\end{quote}

 \noindent
 $\quad${\bf Case ii.2.} Now suppose that $n_b+n_{b-1}>r-t$. Then by definition of $k$ and $\ell$ in (ii),we have $k=r-t$ and $\ell=n_b+n_{b-1}-r+2t.$ So in order to prove that $(k,\ell)\in B_{t}$ in this case, we need to show that $t< \ell\leq u-r$.
 
 Since $n_b+n_{b-1}>r-t$ it is obvious that $t<\ell.$ On the other hand, since $t=n_a+n_{a-1}$, using (\ref{longbottom}) we get $u=bn_b+(b-1)n_{b-1}+2t$.  If $b\geq 2$, then $bn_b+(b-1)n_{b-1}\geq n_b+n_{b-1}$. On the other hand, if $b=1$ then $n_{b-1}=0$, and therefore $bn_b+(b-1)n_{b-1}=n_b=n_b+n_{b-1}$. So in any case, $u\geq n_b+n_{b-1}+2t.$
Therefore, $\ell=n_b+n_{b-1}-r+2t\leq u-r$, as desired.
This shows that $(k,\ell)\in B_t$. The equations \eqref{longbottom}, \eqref{shorttop}, $t=n_a+n_{a-1},k=r-t$, and $\ell=n_b+n_{b-1}-r+2t$ imply that $P_{k,\ell}=P$. This
completes the proof of Claim~B and of Theorem \ref{table2thm}.\vskip 0.2cm
 \end{proof}
\section{The Box Conjecture.}\label{boxdiagsec}
We first recall P. Oblak's Recursive Conjecture for $\mathfrak Q(P)$  and summarize results about it in Section~\ref{recursiveconjsec}.  In Section~\ref{boxconjsubsec} we state a Box Conjecture for $\mathfrak Q^{-1}(Q)$ which is a generalization of Theorems \ref{tablethm} and \ref{table2thm}. 

 In
 Section~\ref{boxspecialsec} we prove the analog of Theorem \ref{tablethm} in the special case that $Q=(u+s,u,u-r), $ with $r\ge 2$ and $2\le s\le 4$.

\subsection{Recursive Conjecture for ${\mathfrak Q}(P)$.}\label{recursiveconjsec}
 P. Oblak conjectured a recursive process for determining $\mathfrak Q(P)$ from $P$ that greatly influenced further work in the area (\cite{BIK,BKO,IK,Kh1,KO}). 
 \par

Recall from Definition \ref{Udef} that a U-chain $C_a$ of $\mathcal D_P$ is comprised of three parts: first a chain through all the vertices in the rows of length $a,a-1$, corresponding to the almost rectangular subpartition $(a^{n_a}, (a-1)^{n_{a-1}})$ of $P$;  then two chains linking those rows to the source and to the sink in the top row of $\mathcal D_P$.
Recall also that the length $|C_a|$ satisfies equation \eqref{Caeq}:
$|C_a|=an_a+(a-1)n_{a-1}+2\sum_{i>a} n_i.$\par

Given a partition $P$ of $n$ and an integer $a\in S_P$ we denote by $P'(P,a)$ the unique partition of $(n-|C_a|)$ obtained by omitting the vertices of  the chain $C_a$ from $\mathcal D_P$ and counting the vertices left 
 in each row. When $P=(\cdots ,i^{n_i},\cdots )$ we have that $P'(P,a)= (\cdots , i^{n'_i},\cdots )$ where the multiplicity integers $n'_i$ satisfy
 \begin{equation}
 n'_i=\begin{cases}&n_i \text { if } i\le  a-2\\
 &n_{i+2} \text { if } i\ge a-1. 
 \end{cases}
 \end{equation}\noindent
 {\it Notation:} We will write $P\vdash n$ for ``the set $P$ of positive integers is a partition of $n$.''\par
 For example, when $P=(7,5,4,3,3,2,1)\vdash 25$ and $a=3$ then $P'(P,3)=(5,3,2,1)\vdash 11$.
 The poset $\mathcal D_{P'}$ is in general not a subposet of $\mathcal D_P$ \cite{BIK,IK,Kh1}. The following recursive process constructs a 
partition $Ob(P)$ of $n$ from a given partition $P$ of $n$.
\begin{definition}[P. Oblak's recursive process]\label{Oblakrecursivedef} Suppose $P$ is a partition of $n$. Let $C_a$ be a U-chain in $\mathcal D_P$ of maximum length, and suppose that $Ob(P')$ where $ P'=P'(P,a)$ has been chosen. Then we set $Ob(P)=(|C_a|,Ob(P'))$. When $P$ is almost rectangular we take $Ob(P)=(n)$.
\end{definition}
As just stated the partition $Ob(P)$ is a priori not well defined, since there may be different possibilities for the choices of maximum length U-chains. Originally, P.~Oblak chose the largest integer $a$ giving a maximum length U-chain in each  step \cite{BKO}. \par
  Several authors associate a partition $\lambda(\mathcal P)$ to any finite poset $\mathcal P$ by first setting $c_i$ equal to the maximum number of vertices covered by $i$ chains of $\mathcal P$, then setting $\lambda_i(\mathcal P)=c_i-c_{i-1}$, with $c_0=0$ (see \cite{BrFo,Gans,Gre,GreKl,Pol, Sak}). For the poset $\mathcal D_P$  the second author defined the partition $\lambda_U(P)=\lambda_U(\mathcal D_P)$ using U-chains in a similar way, setting $c_{i,U}(\mathcal D_P) $ equal to the maximum number of vertices covered by $i$ U-chains, then setting $\lambda_{i,U}(P)=c_{i,U}(\mathcal D_P)- c_{i-1,U}(\mathcal D_P)$. She then showed:
 \begin{theorem}(\cite[Theorem 2.5]{Kh1})\label{indthm} The partition $Ob(P)$ is independent of the choices of maximum length U-chains in Definition \ref{Oblakrecursivedef}, and is equal to $\lambda_U(P)$.
 \end{theorem}
 \begin{conjecture}[Oblak Recursive Conjecture]\label{Oblakrecursconj}
The map $P\to \mathfrak Q(P)$ satisfies $\mathfrak Q(P)=Ob(P)$.
 \end{conjecture}
 It follows from Definition \ref{posetdef} that the poset $\mathcal D_P$ is independent of $\cha {\sf k}$. Since $Ob(P)$ is a combinatorial invariant of $\mathcal D_P$ the Recursive Conjecture implies that $\mathfrak Q(P)$ is independent of $\cha {\sf k}$. Also, by definition $\lambda (\mathcal D_P)\ge \lambda_U (P)$ in the Bruhat order.\footnote{The Bruhat order on partitions $P=(p_1,p_2,\ldots, p_t), P'=(p'_1,\ldots p'_{t'})$ of $n$ is  $
P\ge P' \text { if for all } i, \,\sum_{k=1}^i p_k\ge \sum_{k=1}^i p'_k$.}
A general result due to E.R. Gansner  and M. Saks (\cite{Gans,Sak}, see \cite[Theorem 6.1]{BrFo}) shows that $\lambda(\mathcal D_P)\ge \mathfrak Q(P)$.
 The first and second author showed:
 \begin{theorem}\cite[Theorem 3.9]{IK}\label{halfOblakthm} Let $\sf k$ be an infinite field.  Then $\mathfrak Q(P)\ge \lambda_U(P)$.
 \end{theorem}
 L. Khatami studied the smallest part of $\mathfrak Q(P)$ and defined a somewhat subtle combinatorial invariant  $\mu (P)$
 \cite[Definition 2.6]{Kh2}.  Using a study of the antichains of $\mathcal D_P$ she showed
 \begin{theorem}\cite[Theorem 4.1]{Kh2}\label{minpartthm} Let $P$ be a partition of $n$ and let $\sf k$ be an infinite field. The three partitions $\lambda (\mathcal D_P), \lambda_U(P)$ and $ \mathfrak Q(P)$ have the same smallest part, which is equal to $\mu (P)$. 
 \end{theorem}
 Together with P. Oblak's Index Theorem \ref{indexthm}, this implies the following fact we will use in Section~\ref{boxspecialsec}.		
 \begin{theorem}[Oblak conjecture for $r_P\le 3$]\label{Oblak3thm}
 The Oblak Recursive Conjecture \ref{Oblakrecursconj} is true over any infinite field $\sf k$ when $r_P\le 3$.
 \end{theorem}
 \begin{remark}[Summary of results on the Oblak Recursive Conjecture]\label{historyrem}
 Thus, the cases $r_P=2$ \cite{BIK,KO,Obl1,Z} and $r_P=3$ \cite{Kh2} of the Conjecture have been known since 2008 and 2012, respectively. Theorem~\ref{halfOblakthm} of the first and second authors then showed ``half'' the Conjecture in all characteristics.  Since $\lambda (D_P)\ge \mathfrak Q(P)\ge \lambda_U(P)$ a proof of the purely combinatorial statement $\lambda(\mathcal D_P)=\lambda_U(\mathcal D_P)$ would show the Oblak conjecture for $P$ in all characteristics. In contrast, which pairs of Jordan types occur for $A,B$ with $[A,B]=0$ depends on $\cha {\sf k}$: see \cite[Example 2.18]{BIK}, \cite{BrWi}, and \cite[Example 22]{McN}.
\end{remark}
\begin{lemma}\label{lambdaUstable} \cite[Proposition 2.7]{Kh1} Let $P$ be a partition of $n$. Then $\lambda_U(P)$ has parts that differ pairwise by at least two. 
\end{lemma}
Recall that $\mathfrak Q(P)$ has parts that differ pairwise by at least two when $\cha k=0$ or $\cha k=p>n$ and $\sf k$ is infinite by Theorem \ref{stablethm}.
It follows from Theorems \ref{indexthm} and \ref{minpartthm}, and Lemma~\ref{lambdaUstable} that $\mathfrak Q(P)$ has parts that differ pairwise by at least two over \emph{any} infinite field $\sf k$ when $r_P\le 3$ (since then  $\lambda(\mathcal D_P)=\mathfrak Q(P)=\lambda_U (\mathcal D_P)$).

\subsection{Key of a stable partition $Q$ and the Box Conjecture.}\label{boxconjsubsec}
We first define the \emph{key} of $Q$, which determines the shape of the box $\mathcal B(Q)$ of partitions which conjecturally make up $\mathfrak Q^{-1}(Q)$.
\begin{definition}[Key of $Q$]\label{keydef}  Let $Q=(q_1,q_2,\ldots ,q_k)$ with $ q_1\ge q_2\ge \cdots\ge q_k>0$ be a partition of $n=\sum_{i=1}^k q_i$ such that 
$q_i-q_{i+1}\ge 2$ for $ i\in {1,\ldots, k-1}$. Put
\begin{equation}\label{key1eq}
s_i=\begin{cases}\, q_i-q_{i+1}-1& \text { for } 1\le i\le k-1\\
\, q_k &\text { for } i=k.
\end{cases}
\end{equation}
We call the sequence 
\begin{equation}\label{key2eq}
{\sf S}_Q=(s_1,s_2,\ldots ,s_k)
\end{equation}
the \emph{key} of the stable partition $Q$.
\end{definition}
\begin{example} The key of $Q=(u,u-r)$ is ${\sf S}_Q=(r-1,u-r)$. The key of $Q=(11,6,2)$ is ${\sf S}_Q=(4,3,2)$.
\end{example}
Evidently, the assignment $Q\rightarrow {\sf S}(Q)$ is a bijection between the set of partitions with $k$ parts that differ pairwise by at least two and $\mathbb Z_{>0}^k$. It is easy to see that the inverse to \eqref{key1eq} is
\begin{align}
q_i&=\left(\sum_{j\ge i}s_j\right)+k-i,\notag \\
n&=\left(\sum i\cdot s_i\right)+\frac{k(k-1)}{2}.\label{nSeq}
\end{align}
We now state a conjecture generalizing the Table Theorem (i.e. Theorems \ref{tablethm} and \ref{table2thm}). 
\begin{conj}[Box Conjecture]\label{Zgenconj} Let $Q$ be a partition having $k$ parts that differ pairwise by at least two, and assume that the key of $Q$ is ${\sf S}_Q = (s_1,s_2,\ldots,s_k)$.
\begin{enumerate}[(a)]
\item There is an  $s_1\times s_2\cdots \times s_k$ array (``box'') $\mathcal B(Q)$ of distinct partitions 
\begin{equation}
\mathcal B(Q)=\{P_{i_1,i_2,\ldots, i_k}\mid 1\le i_u\le s_u\},
\end{equation}
 such that $P_{i_1,i_2,\ldots, i_k}$ has $\sum_{1\le u\le k} i_u$ parts and satisfies $\mathfrak Q(P_{i_1,i_2,\ldots , i_k})=Q$.
\item  The cardinality $|\mathfrak Q^{-1}(Q)|=\prod_{1\le i\le k} s_i.$ Equivalently, given (a), the set of partitions in $\mathcal B(Q)$ is the complete set $\mathfrak Q^{-1}(Q)$.
\end{enumerate}
\end{conj}
\begin{remark}
In principle Oblak's Recursive Conjecture \ref{Oblakrecursconj}  for the map $P\to \mathfrak Q(P)$ should allow us to decide the Box Conjecture.  We follow this strategy in Section \ref{boxspecialsec} to prove part (a) of the Box Conjecture for certain $Q$. Of course, a deeper understanding of $\mathcal B(Q)$ and Conjecture \ref{Zgenconj} could very well give a new approach to showing Oblak's Recursive Conjecture. 
In \cite[Section 5.3]{IKVZ} we show that the count of partitions in the box $\mathcal B(Q)$ is the same as that for partitions of diagonal hook lengths given by $Q$: this shows that the Box Conjecture is consistent with a count of the number of partitions of $n$. It would be of interest to show that there is a bijection -- preferably explicit -- between the set of partitions in $\mathfrak Q^{-1}(Q)$ and the set of partitions having diagonal hook lengths $Q$.
\end{remark}
\begin{remark}
Let $Q$ be a partition as in Conjecture \ref{Zgenconj} and $B=J_Q$, the Jordan matrix of partition $Q$ (notation before Lemma \ref{partsilem}). Recall from equation \eqref{locidefeqn} that by the \emph{locus} $\mathfrak Z_P$ of a Jordan type $P\in \mathcal B(Q)$ we mean the Zariski closure of the set $\{A\in \mathcal N_B\mid P_A=P\}$ of matrices commuting with $B=J_Q$ and having Jordan type $P$. Recall that here $\mathcal N_B=\mathcal U_B$ is an affine space (Proposition \ref{ubnbprop} (b)). Together with M. Boij we conjecture that the codimension of the locus $\mathfrak Z_P, P=(P_{i_1,i_2,\ldots ,i_k})$ in $\mathcal N_B$ is $(\sum_{u=1}^k i_u)-k$, and that the locus is a complete intersection defined by explicit irreducible equations of degree at most $k$ in the coordinates of $\mathcal N_B$.
 In a sequel in progress joint with M.~Boij, we plan to show this conjecture for $k=2$. For some more information about these loci, see \cite[\S 4]{IKVZ}.
\end{remark}
Note that when $s_i=1$ there is no contribution of this part of the key to the conjectured cardinality $|\mathfrak Q^{-1}(Q)|$, a fact that was known at least in the case when ${\sf S}_Q=( 1,1,\ldots ,1 ,s_k)$ (see \cite[Theorem 4.1]{Obl2}).  Since the cases where $k=3$ and $s_1,s_2$, or $s_3$ is equal to $ 1$ are relatively easy, we give below several examples of stable partitions $Q$ with three parts for which we have verified\footnote{We used the Oblak-Khatami Theorem \ref{Oblak3thm} and straightforward case-by-case considerations to verify the claims made in Examples \ref{852ex} and \ref{963ex}.} Conjecture \ref{Zgenconj} and where no $s_i$ is equal to $1$.
\begin{example}\label{852ex} Let $Q=(8,5,2)\vdash 15$ so $ {\sf S}_Q=(2,2,2)$. Then  $|\mathfrak Q^{-1}(Q)|=8$. The two floors of $\mathcal B(Q)$ are
\begin{equation}\label{852eq}
\left(\begin{array}{cc}
(8,5,2)& (8,5,1^2)\\
(8,4,[3]^2)&(8,4,[3]^3)
\end{array}\right),
\qquad
\left(\begin{array}{cc}
{(7,4,[4]^2)}& (7,[6]^2,[2]^2)\\
(7,4,[4]^3)&(7,4,[4]^4)
\end{array}\right).
\end{equation}
The floor at left are the partitions obtained from $\mathfrak Q^{-1}((5,2))$ by adjoining the part $8$. The partitions in the second floor at right are obtained by adjoining $7$ to those partitions $P'$ in $\mathfrak Q^{-1}((6,2))$ having no part 6.
 
\end{example}
\begin{example}\label{963ex}
Let $Q=(9,6,3)\vdash 18,$ so $ {\sf S}_Q=(2,2,3)$. Then  $|\mathfrak Q^{-1}(Q)|=12$. The two floors of $\mathcal B(Q)$ are
\begin{equation}\label{963beq}
\left(\begin{array}{ccc}
(9,6,3)& (9,6,[3]^2)&(9,6,[3]^3)\\
(9,5,[4]^2)&(9,5,[4]^3)&(9,5,[4]^4)
\end{array}\right),
\quad
\left(\begin{array}{ccc}
{(8,5,[5]^2)}& (8,[7]^2,[3]^2)&(8,[7]^2,[3]^3)\\
(8,5,[5]^3)&(8,5,[5]^4)&(8,5,[5]^5)
\end{array}\right).
\end{equation}
The two other partitions whose keys are permutations of ${\sf S}_Q$  are $(9,5,2)\vdash 16$ corresponding to key $(3,2,2)$ and $(9,6,2)\vdash 17$ corresponding to key $(2,3,2)$.
For $Q=(9,6,2)\vdash 17$ the array $\mathcal B(Q)$ has the following two floors:
\begin{equation}\label{963ceq}
\left(\begin{array}{cc}
(9,6,2)& (9,6,1,1)\\
(9,4,2^2)&(9,3^2,1^2)\\
(9,4,2,1^2)&(9,4,1^4)
\end{array}\right),
\qquad
\left(\begin{array}{cc}
{ (8,4,3,2)}& (8,4,3,1^2)\\
(8,4,2^2,1)&(8,3^2,1^3)\\
(8,4,2,1^3) &(8,4,1^5)
\end{array}\right).
\end{equation}
For $Q=(9,5,2)\vdash 16$ the array $\mathcal B(Q)$ has these floors:
\begin{equation}\label{952eq}
\left(\begin{array}{cc}
(9,5,2)& (9,5,1,1)\\
(9,4,2,1)&(9,4,1^3)
\end{array}\right),
\quad
\left(\begin{array}{cc}
{(7,4,3,2)}& (7,4,3,1^2)\\
(6,[7]^2,[3]^2)&(6,[7]^2,[3]^3)
\end{array}\right),\quad
\left(\begin{array}{cc}
(7,4,2^2,1)&(7,3^2,1^3)
\\
(7,4,2,1^3)&(7,4,1^5)
\end{array}\right).
\end{equation}
\end{example}
We now give the box $\mathcal B(Q)$ for the simplest example with no $s_i=1$ such that $Q$ has four parts. We have not shown completeness of the box in this case.
\begin{example}\label{11,8,5,2ex}
Let $Q=(11,8,5,2)\vdash 26$, so $ {\sf S}_Q=(2,2,2,2)$. Then  $|\mathfrak Q^{-1}(Q)|\ge 16$. To initially write down the 16-element box $\mathcal B(Q)$,  conveniently viewed with the $4$-D glasses supplied to the reader, we assumed the P. Oblak Conjecture \ref{Oblakrecursconj}, which is open for $r_P=4$. However, we were able to verify that $\mathfrak Q(P)=Q$ for each $P\in \mathcal B(Q)$, using a calculation of the antichains of $\mathcal D_P$ to prove $\lambda(\mathcal D_P)=\lambda_U(\mathcal D_P)$.
 We view $\mathcal B(Q)$ as having two $2\times 2\times 2$ floors. The first floor is obtained by adjoining the part $11$ to each element of $\mathcal B((8,5,2))$ in display \eqref{852eq}. The second floor is 
\begin{equation}\label{11852eq}
\left(\begin{array}{cc}
{ (10,7,4,3,2)}&(10,7,4,2^2,1) \\
(10,7,4,3,1^2)&(10,7,4,2,1^3)
\end{array}\right),
\qquad
\left(\begin{array}{cc}
{(10,6,4,3,2,1)}& (10,7,3^2,1^3)\\
(10,6,4,3,1^3)&(10,7,4,1^5)
\end{array}\right).
\end{equation}
\noindent\end{example}

\subsection{Box Conjecture for certain partitions $Q$ with three parts.}\label{boxspecialsec}
In this section we use a method similar to certain steps in the proof of our main Theorem~\ref{tablethm} to show part (a) of the Box Conjecture-- ``filling the box'' -- for a few (infinite) families of partitions $Q$ with three parts.
Recall from Theorem \ref{Oblak3thm} and Lemma \ref{lambdaUstable} that if $r_P\le 3$ then $Q=\mathfrak Q(P)$ has parts that differ pairwise by at least two and are obtained from $P$ by the Oblak recursive process.
\begin{theorem}\label{fillthm} Let $Q=(u+s,u,u-r)$ with  $ 2\le s\le 4$ and $2\le r$. There is an array $\mathcal B(Q)$ of  dimensions  $(s-1)\times (r-1)\times (u-r)$ of partitions
\begin{equation}\label{box3thm}
 \mathcal B(Q)=\{P_{j,k,\ell} \mid 1\le j\le s-1,1\le k\le r-1,1\le \ell\le u-r\} 
 \end{equation}
 such that $\mathfrak Q(P_{j,k,\ell})=Q$ and $P_{j,k,\ell}$ has $j+k+\ell$ parts.
\end{theorem}
This theorem covers all keys ${\sf S}_Q=(s-1,r-1,u-r)$ with $1\le s-1\le 3$.  Like Theorem~\ref{tablethm} we prove it using Oblak's recursive process. We note that Theorem \ref{fillthm} confirms part (a) of Conjecture \ref{Zgenconj} for the partitions $Q$ under consideration,  but does not show part (b).
We split Theorem \ref{fillthm} into three cases, $s=2$ in Lemma \ref{s=2lem}, $s=3$ in Lemma \ref{s=3lem}, and $s=4$ in Proposition \ref{boxspecialprop}. 
\begin{lemma}\label{s=2lem}
Let $Q=(u+2, u, u-r)$ with $u>r\geq 2$. For $k\in \{1,\ldots ,r-1\} $ and $\ell\in \{1,\ldots ,u-r\}$ put
\begin{equation*}
P_{1,k,\ell}=(u+2,P_{k,\ell}),
\end{equation*}
where $P_{k,\ell}$ is the partition in $\mathcal T((u,u-r))$ defined in Theorem \ref{tablethm}. Then $\mathfrak Q(P_{1,k,\ell})=Q$ and $P_{1,k,\ell}$ has $1+k+\ell$ parts for all $k\in \{1,\ldots ,r-1\} $ and $\ell\in \{1,\ldots ,u-r\}$. 
\end{lemma}
\begin{proof}
For all $P_{k, \ell} \in \mathcal{T}((u,u-r))$, the largest part of $P_{k, \ell}$ is at most $u$ so it differs from $u+2$ by at least 2. Thus, the only almost rectangular subpartition of $P_{1,k,\ell}$ that includes $u+2$ is $(u+2)$ itself. On the other hand, since $P_{k,\ell}\in \mathcal{T}((u,u-r))$, we have $\mathfrak{Q}((P_{k, \ell}))=(u,u-r)$. Thus, the longest U-chain in the poset of $P_{1,k,\ell}$ has length $u+2$, and can be obtained from the top part $(u+2)$ or the union of the longest U-chain in the poset of $P_{k,\ell}$ and the first and last vertices in the row representing $u+2$ in the poset of $P_{1,k,\ell}$. Choosing the top part $(u+2)$ as the longest $u$-chain, it follows from Theorem \ref{Oblak3thm} that $\mathfrak{Q}((u+2, P_{k, \ell}))=Q$.
 The claim about the number of parts of $P_{1,k,\ell}$ is immediate from Theorem \ref{tablethm}(d).
\end{proof}
\begin{lemma}\label{samelargestpartlem} Let $Q'=(u,u-r)$ with $u>r\ge 2$. The largest part of $P_{2,1}\in \mathcal T(Q')$ is greater than or equal to the largest part of each partition $P_{k,\ell}\in \mathcal T(Q')$ with $k\ge 2$.
\end{lemma}
\begin{proof}  Let $P_{k,\ell}\in \mathcal T(Q')$. If $P_{k,\ell}$ is of type A then by Theorem \ref{tablethm} its largest part is the largest part of $[u]^k$, which is $\lceil \frac{u}{k}\rceil$. If $P_{k,\ell}$ is of type B or type C, then its largest part is the largest part of $[u-r+2t]^t$, which is  $\lceil \frac{u-r}{t}\rceil+2$, for some $t\in \{1,2,\ldots,t_\max\}$. Let  $\delta$ be the largest part of $P_{2,1}$.
It is enough to show that
\begin{align}\label{deltaAeqn}
\delta &\ge \lceil \frac{u}{k}\rceil \text { for all $k\ge 2$; and}\\
\delta&\ge \lceil \frac{u-r}{t}\rceil+2 \text { for all } t\in \{1,2,\ldots, t_\max\}\label{deltaBeqn}.
\end{align}
By Lemma \ref{tableinvariantslem}(b) we know that $k_1=\lceil\frac{u}{u-r+1}\rceil\ge 2$.  We consider two cases.
\begin{enumerate}[a.]
\item If $k_1=2$ then $P_{2,1}$ is of type B or C and its largest part is $\delta=u-r+2$. Inequality \eqref{deltaBeqn} is clear. Since $k_1=2$ we know that $\frac{u}{u-r+1}\le 2$, which implies that 
$u-r+2\ge \frac{u}{2}+1$. This shows that $\delta>\lceil \frac{u}{2}\rceil$ and inequality \eqref{deltaAeqn} follows. 
\item If $k_1>2$ than $P_{2,1}$ is of type A and $\delta=\lceil \frac{u}{2}\rceil $. Inequality \eqref{deltaAeqn} follows immediately.
Since $k_1>2$, we have $\frac{u}{u-r+1}>2$ and therefore
$\frac{u}{2}>u-r+1$. This implies that $\delta\ge u-r+2$, and inequality \eqref{deltaBeqn} follows.
\end{enumerate}\end{proof}
\begin{lemma}\label{s=3lem}
Let $Q=(u+3, u, u-r)$ with $u>r\geq 2$. For $k\in \{1,\ldots ,r-1\} $ and $\ell\in \{1,\ldots ,u-r\}$ put
\begin{equation*}
P_{1,k,\ell}=(u+3, P_{k,\ell}),
\end{equation*}
where $P_{k,\ell}$ is the partition in $\mathcal{T}((u,u-r))$ defined in Theorem \ref{tablethm}, and put
\begin{equation*}
P_{2,k,\ell}=(u+2, P_{k+1,\ell}),
\end{equation*}
where $P_{k+1,\ell}$ is the partition in $\mathcal{T}((u+1,u-r))$ defined in Theorem \ref{tablethm}.
Then $\mathfrak Q(P_{j,k,\ell})=Q$ and $P_{j,k,\ell}$ has $j+k+\ell$ parts for all $j\in\{1,2\},k\in \{1,\ldots ,r-1\} $ and $\ell\in \{1,\ldots ,u-r\}$
\end{lemma}
\begin{proof}
Since the largest part of every partition in $ \mathcal{T}((u,u-r))$ is at most $u$, it follows from Theorem \ref{Oblak3thm} that 
for all $P_{k,\ell}\in \mathcal{T}((u,u-r))$, we get $\mathfrak{Q}((u+3, P_{k,\ell}))=Q$, as desired.
\\
Now we consider partitions of the form $P=(u+2, P_{k+1,\ell})$ with $P_{k+1,\ell}\in \mathcal T((u+1,u-r))$. We now show that the largest part of $P_{k+1,\ell}\in \mathcal T((u+1,u-r))$ is at most $u$.
First we consider $P_{2,1}\in \mathcal{T}((u+1,u-r))$. Since in  $\mathcal{T}((u+1,u-r))=\mathcal{T}((u+1,u+1-(r+1)))$, $k_1=\lceil\frac{u+1}{(u+1)-(r+1)+1}\rceil=\lceil\frac{u+1}{u-r+1}\rceil$, we have $k_1>2$ if and only if $2r>u+1$. Consequently,
$$
P_{2,1}=\left \{
\begin{array}{lll}
([u+1]^2, u-r)&&\mbox{ (Type A) if }2r>u+1; \\ \\
(u-r+2, [u-2]^2)&&\mbox{ (Type B) if }2r\leq u+1.
\end{array}
\right.
$$
If $2r>u+1$ then $\frac{u+1}{2}<r$ and therefore $\lceil \frac{u+1}{2}\rceil\leq r<u$ so the largest part of $P_{2,1}$ is less than $u$. By Lemma \ref{samelargestpartlem} it follows that all $P_{k+1,\ell}\in \mathcal{T}((u+1,u-r))$ with $k\geq 1$, have  largest part smaller than $u$. 
If $2r \leq u+1$, then the largest part of $P_{2,1}$ is at most $u-r+2\leq u$. It follows again by Lemma \ref{samelargestpartlem}  that the largest part of all $P_{k+1,\ell}\in \mathcal T((u+1,u-r))$ is at most  $u$.\par

Thus, in either case, the longest U-chain in the poset of $P_{2,k,\ell}$ has length $u+1+2=u+3$ and it is the union of the longest U-chain in the poset of $P_{k+1, \ell}$ and the first and last vertices in the $u+2$ row of the poset of  $P_{2,k,\ell}$. Once this U-chain is removed from the poset $P_{2,k,\ell}$  the remaining U-chains have lengths $u$ (left over on top) and $u-r+2$ (the remaining vertices in the poset of $P_{k+1,\ell}$ union the first and last remaining vertices on the top row). Thus, by the Oblak recursive process, $\mathfrak{Q}((u+2, P_{k,\ell}))=(u+3,u, u-r)$, as desired. The assertion about the number of parts of $P_{j,k,\ell}$ is immediate from Theorem \ref{tablethm}(d).
\end{proof}\par
We divide $\mathcal B(Q), Q=(u+4,u,u-r)$ into
3 levels, each comprising a $(r-1)\times (u-r)$ table of partitions, and labelled by the first entry $i$ of $P_{i,k,\ell}$. We now specify the entries of each level.
\begin{proposition}\label{boxspecialprop} Let $Q=(u+4, u, u-r)$ with $u>r\geq 2$. For $k\in \{1,\ldots ,r-1\} $ and $\ell\in \{1,\ldots ,u-r\}$ put
\begin{equation*}
P_{1,k,\ell}: =(u+4, P_{k,\ell})
\end{equation*}
where $P_{k,\ell}$ is the partition in $\mathcal T((u,u-r))$ defined in Theorem \ref{tablethm}; put
\begin{equation*}
P_{2,1,\ell}: =(u+2,P_{2,\ell})
\end{equation*}
where $P_{2,\ell}$ is the partition in $\mathcal T((u+2,u-r))$ defined in Theorem \ref{tablethm}; put
\begin{equation*}
P_{3,k,\ell}: =(u+2, P_{k+2,\ell})
\end{equation*}
where $P_{k+2,\ell}$ is the partition in $\mathcal T((u+2,u-r))$ defined in Theorem \ref{tablethm}.
 For $(k,\ell)$ satisfying $2\le k\le r-1$ and $1\le \ell\le u-r$ put
 \begin{equation*}
 P_{2,k,\ell}: =\begin{cases}  
([u+4]^2, P_{k, \ell})\quad  \text{ if } 2r> u+2,\text {where } P_{k,\ell}\in \mathcal{T}((u,u-r));&\\
([u+4]^2, P_{k,\ell})\quad \text { if $2r\le u+2$ and $2\le k\le r-2$ and $2\le \ell\le u-r$},&\\
\qquad\qquad\qquad\qquad\qquad\text { where } P_{k,\ell}\in \mathcal{T}((u,u-r));&\\
(u-r+4, P_{2, k})\quad \text{ if }2r\leq u+2 \text { and $2\le k\le r-1$ and $\ell =1$ where }P_{2,k}\in \mathcal{T}((u+2,u-2));&\\
(u-r+4,P_{2, r-2+\ell}) \,\, \text { if  $ 2r\le u+2$ and $k=r-1$ and $2\le \ell\le u-r$},&\\
\qquad\qquad\qquad\qquad\qquad\text { where } P_{2,r-2+\ell}\in \mathcal T((u+2,u-2)).&
  \end{cases}
 \end{equation*}
Then $\mathfrak Q(P_{j,k,\ell})=Q$ and $P_{j,k,\ell}$ has $j+k+\ell$ parts for all $j\in\{1,2,3\},\,
k\in\{1,\ldots,r-1\}$ and $\ell\in\{1,\ldots,u-r\}$.
\end{proposition}
\begin{proof} Since the largest part of every partition in $ \mathcal{T}((u,u-r))$ is at most $u$, it follows readily from Theorem \ref{Oblak3thm}  that $\mathfrak Q(P_{1,k,\ell})=Q$ for all $k,\ell$. 
\par
To prove that $\mathfrak Q(P_{2,1,\ell})=Q$ and $\mathfrak Q(P_{3,k,\ell})=Q$ for all $k$ and $\ell$ we begin by claiming that the largest part of $P_{2,1,\ell}$ and of $P_{3,k,\ell}$, namely $u+2$, differs from the second part by at least $2$.\par\noindent
Indeed, since by Lemma \ref{samelargestpartlem} the largest part of every $P_{i,j}\in \mathcal T((u+2,u-r))$ with $i\ge 2$ is at most equal to the largest part of $P_{2,1}\in \mathcal T((u+2,u-r))$, it is enough to compare the largest part of $P_{2,1}$ and $u+2.$  Theorem \ref{tablethm} tells us what the largest part of $P_{2,1}$ is, depending on whether $k_1>2$ (in which case $P_{2,1}$ is of type A), or $k_1=2$ (in which case $P_{2,1}$ is of type B). 
If $k_1=2$ then the largest part of $P_{2,1}$ is $u-r+2$, which is at most $u$ because $r$ is at least 2. On the other hand, if $P_{2,1}$ is of type A then its largest part is $\lceil \frac{u+2}{2} \rceil$. We have 
$$\Big\lceil \frac{u+2}{2} \Big\rceil =  \Big\lceil \frac{u}{2} \Big\rceil+1 \leq \frac{u+1}{2}+1,$$
 which is at most $u$, because $u$ is at least $3$.
In either case the largest part of $P_{2,1}$ is at most $u$. This proves the claim. \par
 It now follows, as in the proof of Lemma \ref{s=3lem}, that the Oblak recursive process implies that $\mathfrak Q(P_{2,1,\ell})=\mathfrak Q(P_{3,k,\ell})=Q$.
\vskip 0.2cm
What's left is to prove that $\mathfrak Q({P}_{2,k,\ell})=Q$ when $2\le k\le r-1$ and $1\le \ell\le u-r$.
If $r=2$, then there are no such partitions and there is nothing to prove. We assume that $r>2$. Consequently, we also have $u>3$.
Recall that $k_1$ for $Q=(u,u-r)$ is at least 3 if and only if $2r>u+2$, and otherwise $k_1=2$. 
\vskip 0.3cm
\underline{Case 1}. Assume that $2r> u+2$.  Then $k_1\ge 3$ for $Q'=(u,u-r)$. 
Therefore,  $P_{2,1}$ is of type A and so $P_{2,1}=([u]^2,u-r)$, and $P_{2,2,1}=([u+4]^2,[u]^2, u-r)$.

Note that the largest part of $([u]^2, u-r)$ is $\lceil \frac{u}{2}\rceil$, and the smallest part of $[u+4]^2$ is $\lfloor \frac{u+4}{2} \rfloor$. So if $u$ is even then the difference is 2, which implies that $\mathfrak{Q}(P_{2,2,1})=Q$, as desired. Now assume that $u$ is odd. Then
 $${P}_{2,2,\ell}=\left([u+4]^2,[u]^2, [u-r]^\ell\right)=\left(\frac{u+1}{2}+2, \frac{u+1}{2}+1, \frac{u+1}{2}, \frac{u-1}{2},[u-r]^\ell\right).$$ 
  One checks that $C_a$, with $a=\frac{u+1}{2}+2$ is one of the maximum-length U-chains
in the poset of ${P}_{2,2,1}$ (of length $u+4$) and then that the Oblak recursive process starting with $C_a$ gives $\mathfrak Q({P}_{2,2,1})=Q$.
  Since by  Lemma  \ref{samelargestpartlem} the largest part of each partition in $\{P_{i, j}\in \mathcal{T}((u,u-r)) \mid i\geq 2\}$ is less than or equal to the largest part of $P_{2,1}$, which is $\frac{u}{2}$ if $u$ is even and $\frac{u+1}{2}$ if $u$ is odd, it follows that $\mathfrak Q({P}_{2,k,\ell})=Q$ for $2\le k\le r-1$ and $1\le \ell\le u-r$, as desired.
\vskip 0.2cm
\underline{Case 2}. Assume that  $2r\leq u+2$, so $k_1=2$ for $Q'=(u,u-r)$ .\par
We first deal with the partitions 
$P_{2,k,\ell}=([u+4]^2,P_{k,\ell})$, with $P_{k,\ell}\in\mathcal T((u,u-r))$ for  $2\le k\le r-2$  and $2\le \ell\le u-r$.  Such partitions only exist if $r\ge 4$. If $r=4$ then $t_{\max}=1$ for $\mathcal T((u,u-r))$. On the other hand, if $r>4$ then the assumption $2r\le u+2$ yields $u-r>2$ and it follows by \eqref{kalternative} that
$k_2\ge 3$ for $\mathcal T((u,u-r))$. Either way, Theorem \ref{tablethm} implies that the partition $P_{2,2}\in \mathcal T(( u,u-r))$ is of type A and its largest part is $\lceil \frac{u}{2}\rceil$.  Since a partition in any B/C hook of $\mathcal T((u,u-r))$ for $t\ge 2$ begins with $[u-r]^t+2$ and since $r\ge 4$ its largest part is no greater than  $\lceil \frac{u}{2}\rceil$. Any partition $P_{k,\ell}\in \mathcal T((u,u-r))$ of type A with $k\ge 2$ has largest part no greater than $\lceil \frac{u}{2}\rceil$.  We now conclude, with the same argument as in Case 1, that $\mathfrak Q(P_{2,k,\ell})=Q$ for all such partitions.
\par
The remaining partitions are ${P}_{2,k,\ell}$ where $k\ge 2$ and $\ell=1$, or where $k=r-1$ and $2\le \ell \le u-r$.  Each of these partitions have the form $(u-r+4, P_{2,j})$ where $P_{2,j}\in \mathcal T((u+2,u-2))$ with $j\ge 2$. Note that since $(u+2)-(u-2)=4$, we have for $\mathcal T((u+2,u-2))$ that $t_{\max}=1$. Thus $P_{2,j}$ with $ j\ge 2$ has type A and satisfies $P_{2,j}=([u+2]^2,[u-2]^j)$.  Its largest part is $\lceil \frac{u+2}{2}\rceil$. From the assumption $2r\le u+2$ we obtain $2(u-r+2)\ge u+2$, whence $\frac{u+2}{2}\le u-r+2$, so $\lceil \frac{u+2}{2}\rceil\le u-r+2$. Thus the largest part of $P_{2,j}$ differs from $u-r+4$ by at least $2$. It now follows, as in the proof of 
Lemma \ref{s=3lem} that the Oblak recursive process gives $\mathfrak Q(P_{2,k,\ell})=Q$ in this case as well.\par

We have shown that $\mathfrak Q(P_{j,k,\ell})=Q$ for all $j\in\{1,2,3\}, k\in \{1,\ldots ,r-1\}, \ell\in\{1,\ldots ,u-r\}$. The assertion about the number of parts of $P_{j.k,\ell}$ follows readily by applying Theorem \ref{tablethm}(d).
\end{proof}
\begin{comment}
\begin{table}
$\qquad s_1=1$:
\begin{equation*}
\quad \begin{array}{cc|cc|cc}
&(11,7,3)&&(11,7,[3]^2)&&(11,7,[3]^3)\\
\hline
&(11,5,[5]^2)&&(11,[7]^2,[3]^2)&&(11,[7]^2,[3]^3)\\
\hline
&(11,5,[5]^3)&&(11,5,[5]^4)&&(11,5,[5]^5)\\
\end{array}
\end{equation*}
\vskip 0.2cm
$\qquad s_1=2$:
\begin{equation*}
\quad \begin{array}{cc|cc|cc}
&(9,[9]^2,3)&&(9,[9]^2,[3]^2)&&(9,[9]^2,[3]^3)\\
\hline
&(7,[9]^2,[5]^2)&&([11]^2,[7]^2,[3]^2)&&([11]^2,[7]^2,[3]^3)\\
\hline
&(7,[9]^2,[5]^3)&&(7,[9]^2,[5]^4)&&(7,[9]^2,[5]^5)\\
\end{array}
\end{equation*}
\vskip 0.2cm
$\qquad s_1=3$:
\begin{equation*}
\quad \begin{array}{cc|cc|cc}
&(9,5,[7]^3)&&(9,[7]^2,[5]^3)&&(9,[9]^3,[3]^3)\\
\hline
&((9,5,[7]^4)&&(9,[7]^2,[5]^4)&&(9,[7]^2,[5]^5)\\
\hline
&(9,5,[7]^5)&&((9,5,[7]^6)&&(9,5,[7]^7)\\
\end{array}
\end{equation*}
\caption{The box $\mathcal B((11,7,3))$}\label{11,7,3table}
\end{table}
\end{comment}
\begin{figure}[htb]
\begin{center}
\includegraphics[scale=0.63]{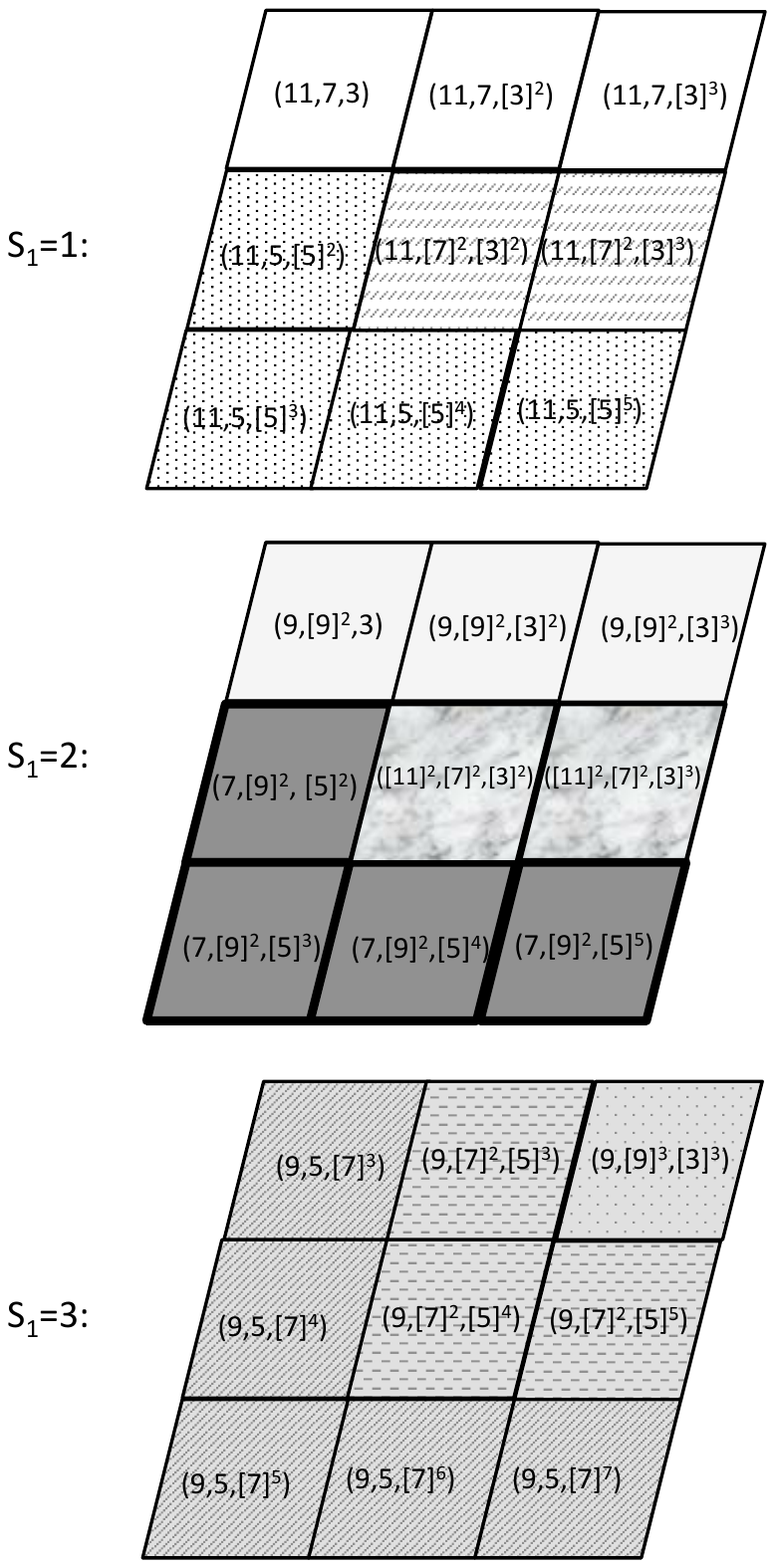}
\end{center}
\vskip -0.21cm
\caption{Box $\mathcal B(Q)$ for $Q=(11,7,3)$}\label{(11,7,3)fig}
\end{figure}

\begin{example}\label{11,7,3ex} In Figure \ref{(11,7,3)fig}, we specify the box $\mathcal B(Q)$ for $Q=(11,7,3)$. Since the key of $Q$ is ${\sf S}_Q=(3,3,3),\, \mathcal B(Q)$ is a $3\times 3\times 3$ cube.  This box is arranged in accordance with Proposition~\ref{boxspecialprop}, and $S_1\in \{1,2,3\}$ is the first index.
 \end{example}

\noindent
{\bf Funding.}
 The third-named author received support from The City University of New York PSC-CUNY Research Award Program and from the National Science Foundation through grant number DMS-1407394.
\begin{ack}
 We are grateful to Polona~Oblak and Toma\v{z} Ko\v{s}ir for their inspiring work on the $P\to \mathfrak Q(P)$ problem, and to Polona Oblak for a second fascinating conjecture in \cite{Obl2} that was the start of our work. The first author is grateful to Toma\v{z} Ko\v{s}ir and his colleagues for an invitation to visit U. Ljubljana in June of 2014, and for discussions he had there.\par
  We greatly appreciate comments of Don King, Alfred No\"{e}l and George McNinch, a discussion several of us had with Barry Mazur, and also discussions several of us had with Eric Friedlander and Julia Pevstova about potential connections with their work on Jordan types. \par
  We are especially grateful for the comments and contribution to the paper of two anonymous referees: one made especially extensive and detailed comments which we much appreciated, and we made many corrections and changes in response.  Another gave a more broad critique which we also appreciated, and as a result we hope we have made the main ideas more accessible, earlier in the paper.
\end{ack}

\vskip 0.4cm
{\small\sc Department of Mathematics, Northeastern University, Boston MA 02115,
 USA}\\{\small{\it E-mail address:}  
a.iarrobino@neu.edu
\vskip 0.2cm
{\sc Department of Mathematics, Union College, Schenectady, NY 12308, USA}\\
{\small{\it E-mail address:} khatamil@union.edu}\vskip 0.2cm 

{\small\sc Department of Mathematics, Medgar Evers College, City University of New York, Brooklyn, NY 11225, USA }\\
{\it E-mail address:} bartvs@mec.cuny.edu
\vskip 0.2cm

{\small\sc Mathematics Department, University of Missouri, Columbia, MO, 65211, USA }\\
{\it E-mail address:} zhaorui0408@gmail.com


\begin{thebibliography}{ACGHO}
\renewcommand{\baselinestretch}{1.7}

\bibitem[1]{Bar}
V. Baranovsky: \emph{The variety of pairs of commuting nilpotent matrices is 
irreducible}, Transform. Groups  6 (1) (2001) 3--8.

\bibitem[2]{Bas}
R. Basili: \emph{On the irreducibility of commuting varieties of nilpotent  matrices},
J.~Algebra  268 (1) (2003) 58--80.

\bibitem[3]{Bas2}
R. Basili: \emph{On the maximum nilpotent orbit intersecting a centralizer in $M(n,K)$}, preprint, 2014, arXiv:1202.3369 v.5 (math.RT).

\bibitem[4]{BI}
R. Basili and A. Iarrobino: \emph{Pairs of commuting nilpotent matrices, and Hilbert function},
J. Algebra  320 (3) (2008) 1235--1254.

\bibitem[5]{BIK}
R. Basili, A. Iarrobino and L. Khatami, \emph{Commuting nilpotent matrices and Artinian Algebras}, J. Commutative Algebra 2 (3) (2010) 295--325.

\bibitem[6]{BKO}
R. Basili, T. Ko\v{s}ir, P. Oblak: \emph{Some ideas from Ljubljana}, (2008), preprint.

\bibitem[7]{BrWi}
J.R. Britnell and M. Wildon: \emph{On types and classes of commuting matrices over finite fields}, J. Lond. Math. Soc. (2) 83 (2) (2011)  470--492.

\bibitem[8]{BrFo}
T. Britz and S. Fomin: \emph{Finite posets and Ferrers shapes}, Advances Math. 158 (1) (2001) 86--127.

\bibitem[9]{BrBr}
J. Brown and J. Brundan: \emph{Elementary invariants for centralizers of nilpotent matrices}, J. Aust. Math. Soc. 86 (1) (2009) 1--15.

\bibitem[10]{BuEv}
M. Bulois, L. Evain: \emph{Nested punctual Hilbert schemes and commuting varieties of parabolic subalgebras}, J. Lie Theory 26 (2) (2016) 497--533.

\bibitem[11]{CM}
D. Collingwood, W. McGovern: \emph{Nilpotent Orbits in Semisimple Lie algebras}, Van Nostrand Reinhold (New York), 1993.

\bibitem[12]{DGKO} 
G. Dolinar, A. Guterman, B. Kuzma, P. Oblak: \emph{Extremal matrix centralizers}, Linear Algebra Appl. 438 (7) (2013)  2904--2910.

\bibitem[13]{DrKi}
Y. Drozd, V. Kirichencko: \emph{Finite Dimensional Algebras}, Springer-Verlag (Berlin) (1994).

\bibitem[14]{FPS}
E. Friedlander, J. Pevtsova, A. Suslin: \emph{Generic and maximal Jordan types}, Invent. Math. 168 (3) (2007) 485--522.

\bibitem[15]{Gans}
E.R. Gansner: \emph{Acyclic digraphs, Young tableaux and nilpotent matrices},  SIAM Journal on Algebraic Discrete Methods 2 (4) (1981) 429--440.

\bibitem[16]{Gi}
V. Ginzburg: \emph{Principal nilpotent pairs in a semisimple Lie algebra. I}, Invent. Math. 140 (3) (2000) 511--561.

\bibitem[17]{Gre}
C. Greene: \emph{Some partitions associated with a partially ordered set},
J. Combinatorial Theory Ser A 20 (1) (1976) 69--79.

\bibitem[18]{GreKl}
C. Greene and D. Kleitman: \emph{The structure of Sperner k-families},
J. Combinatorial Theory Ser. A  20 (1) (1976) 41--68. 
  
\bibitem[19]{GurSe}
R. Guralnick and B.A. Sethuraman: \emph{Commuting pairs and triples of matrices and related varieties}, Linear Algebra Appl. 310 (1-3) (2000) 139--148.

\bibitem[20]{HaHy}
W. Haboush and D. Hyeon: \emph{Conjugacy classes of commuting nilpotents}, preprint, 2016, arXiv:1606.09625  (math.AG).

\bibitem[21]{HW}
T. Harima and J. Watanabe: \emph{The commutator algebra of a nilpotent matrix and an application to the
theory of commutative Artinian algebras},  J. Algebra 319 (6) (2008) 2545--2570.

\bibitem[22]{IK}
A. Iarrobino and L. Khatami: \emph{Bound on the Jordan type of a generic nilpotent matrix commuting with a given matrix}, J. Alg. Combinatorics 38 (4) (2013) 947--972. 

\bibitem[23]{IKVZ}
A. Iarrobino, L. Khatami, B. Van Steirteghem, and R. Zhao: \emph{Nilpotent matrices having a given Jordan type as maximal commuting nilpotent orbit}, preprint, 2015, arXiv:1409.2192 v.2 (math.RA).

\bibitem[24]{J}
N. Jacobson: \emph{Schur's theorems on commutative matrices,} 
Bull. Amer. Math. Soc. 50 (1944) 431--436.

\bibitem[25]{Kh1}
L. Khatami: \emph{The poset of the nilpotent commutator of a nilpotent matrix}, Linear Algebra Appl. 439 (12) (2013) 3763--3776.

\bibitem[26]{Kh2}
L. Khatami: \emph{The smallest part of the generic partition of the nilpotent commutator of a nilpotent matrix}, J.  Pure Appl. Algebra 218 (8) (2014) 1496--1516.

\bibitem[27]{KO}
T. Ko{\v{s}}ir and P. Oblak: \emph{On pairs of commuting nilpotent matrices}, Transform. Groups 14 (1) (2009) 175--182.

\bibitem[28]{Mac} Macaulay F. H. S.: \emph{On a method of dealing
with the intersections of plane curves}, Trans. A.M.S. 5 (4) (1904)
385--410.

\bibitem[29]{Ma}
A. Malcev: \emph{Commutative subalgebras of semisimple Lie algebras}, Izvestia Ak. Nauk USSR
(Russian) 9  (1945), 125--133; English:  Amer. Math. Soc. Translations No. 40  (1951).

\bibitem[30]{McN}
G. McNinch: \emph{On the centralizer of the sum of commuting nilpotent elements}, J.~Pure and Applied Alg. 206 (1-2) (2006) 123--140.

\bibitem[31]{Nak} H. Nakajima: \emph{Lectures on Hilbert schemes of points on surfaces}, University Lecture Series, Amer. Math. Soc.  vol. 18, Providence, RI, 1999.

\bibitem[32]{Ng}
N. Ngo: \emph{On nilpotent commuting varieties and cohomology of Frobenius kernels}, 
J. Algebra 425 (2015), 65--84.  

\bibitem[33]{NgSi}
N. Ngo and K. \v{S}ivic: \emph{On varieties of commuting nilpotent matrices}, Linear Algebra Appl. 452 (2014), 237--262. 

\bibitem[34]{Obl1}
P. Oblak: \emph{The upper bound for the index of nilpotency for a matrix commuting with a
given nilpotent matrix}, Linear and
Multilinear Algebra 56 (6) (2008) 701--711. Slightly revised in
 arXiv:0701561 v.2 [math.AC].

\bibitem[35]{Obl2}
P. Oblak: \emph{On the nilpotent commutator of a nilpotent matrix}, Linear
Multilinear Algebra 60 (5) (2012) 599--612.

\bibitem[36]{Pan}
D. I. Panyushev: \emph{Two results on centralisers of nilpotent elements},
J. Pure and Applied Algebra, 212 (4) (2008), 774--779.

\bibitem[37]{Pan2}
D. I. Panyushev: \emph{Nilpotent pairs, dual pairs, and sheets}, J. Algebra 240 (2) (2001), 635--664.

\bibitem[38]{Pol}
S. Poljak: \emph{Maximum Rank of Powers of a Matrix of Given Pattern}, Proc. A.M.S.
106 (4) (1989) 1137--1144.

\bibitem[39]{Prem}
A. Premet: \emph{Nilpotent commuting varieties of reductive Lie algebras}, Invent.
Math. 154 (3) (2003) 653--683.

\bibitem[40]{Sak}
M. Saks:  \emph{Some sequences associated with combinatorial structures},
Discrete Math. 59 (1-2) (1986) 135--166.

\bibitem[41]{Si}
K. \v{S}ivic:  \emph{On varieties of commuting triples II}, Linear Algebra Appl. 437 (2) (2012) 461--489.

\bibitem[42]{SuTy} 
D. A. Suprunenko, R.I. Tyshkevich: \emph{Commutative Matrices}, viii+155p. Academic Press, New York, 1968.

\bibitem[43]{SusFB1}
A. Suslin, E.M. Friedlander, and C.P. Bendel: \emph{Infinitesimal 1-parameter subgroups and cohomology}, J. Amer. Math. Soc. 10 (3) (1997) 693--728.

\bibitem[44]{SusFB2}
A. Suslin, E.M. Friedlander, and C.P. Bendel: \emph{Support varieties for infinitesimal group schemes}. J. Amer. Math. Soc. 10 (3) (1997) 729--759.

\bibitem[45]{TA} 
H.W. Turnbull and A.C. Aitken: \emph{An Introduction to the Theory of Canonical Matrices}, Dover, New York, 1961.

\bibitem[46]{Z}
R. Zhao: \emph{Commuting nilpotent matrices and normal patterns in Oblak's proposed formula}, preprint, 2014.

\end{thebibliography}
\end{document}